\documentclass[10pt,leqno]{amsart}
\usepackage{graphicx}
\usepackage{indentfirst,csquotes}

\usepackage{amssymb,amsthm,amsmath}
\usepackage{xcolor,paralist,hyperref,titlesec,fancyhdr,etoolbox}
\newtheorem{theorem}{Theorem}[]

\newtheorem{lemma}[theorem]{Lemma}

\titleformat{\section}[display]{\normalfont\huge\bfseries\centering}{\centering\chaptertitlename\thechapter}{10pt}{\Large}
\titlespacing*{\section}{0pt}{0ex}{0ex}

\hypersetup{ colorlinks=true, linkcolor=black, filecolor=black, urlcolor=black }

\usepackage{lipsum}

\begin{document}
\title{Apollonius Problem and Caustics of an Ellipsoid} 
\author[Y.N. Aliyev]{Yagub N. Aliyev}
\date{\today}
\address{ADA University, Ahmadbey Aghaoglu str. 61, Baku, 1008, Azerbaijan}
\email{yaliyev@ada.edu.az}
\maketitle

\let\thefootnote\relax
\footnotetext{MSC2020: Primary 53A05, Secondary: 53A04.

Keywords: ellipsoid; Apollonius; caustics; centro-surface; astroida; ellipse.}
\begin{abstract}
In the paper we discuss Apollonius Problem on the number of normals of an ellipse passing through a given point. It is known that the number is dependent on the position of the given point with respect to a certain astroida. The intersection points of the astroida and the ellipse are used to study the case when the given point is on the ellipse. The problem is then generalized for 3-dimensional space, namely for Ellipsoids. The number of concurrent normals in this case is known to be dependent on the position of the given point with respect to caustics of the ellipsoid. If the given point is on the ellipsoid then the number of normals is dependent on position of the point with respect to the intersections of the ellipsoid with its caustics. The main motivation of this paper is to find parametrizations and classify all possible cases of these intersections.
\end{abstract} 

\bigskip

$\,$

$\,$
\textbf{1. Introduction}

How many normals can one draw from a point to an ellipse? In the current paper we will try to solve this problem and its generalization to 3 dimensions, using the methods of differential and integral calculus and differential geometry, which were not around when Apollonius of Perga (c. III-II centuries BC) first asked and answered this question in his famous work \textit{Conics} \cite{apol}. Their number is not the only interesting question about these normals.
For example, theorem proved by Joachimstal in 1843 states that if $AB_1$, $AB_2$, $AB_3$, and $AB_4$ are these normals, then the points $B_1,\ B_2,\ B_3$, and the point diametrically opposite to $B_4$, with respect to the center $O$, of the ellipse, are concyclic \cite{joach2} (see also Sect. 17.2 in \cite{berger2},  \cite{hamf}, \cite{cayley2}, \cite{desbov}, \cite{legendre}). There are more results related to this fact in \cite{berger}, Sect. 17.7.3. 

\begin{figure}[htbp]
\centerline{\includegraphics[scale=.2]{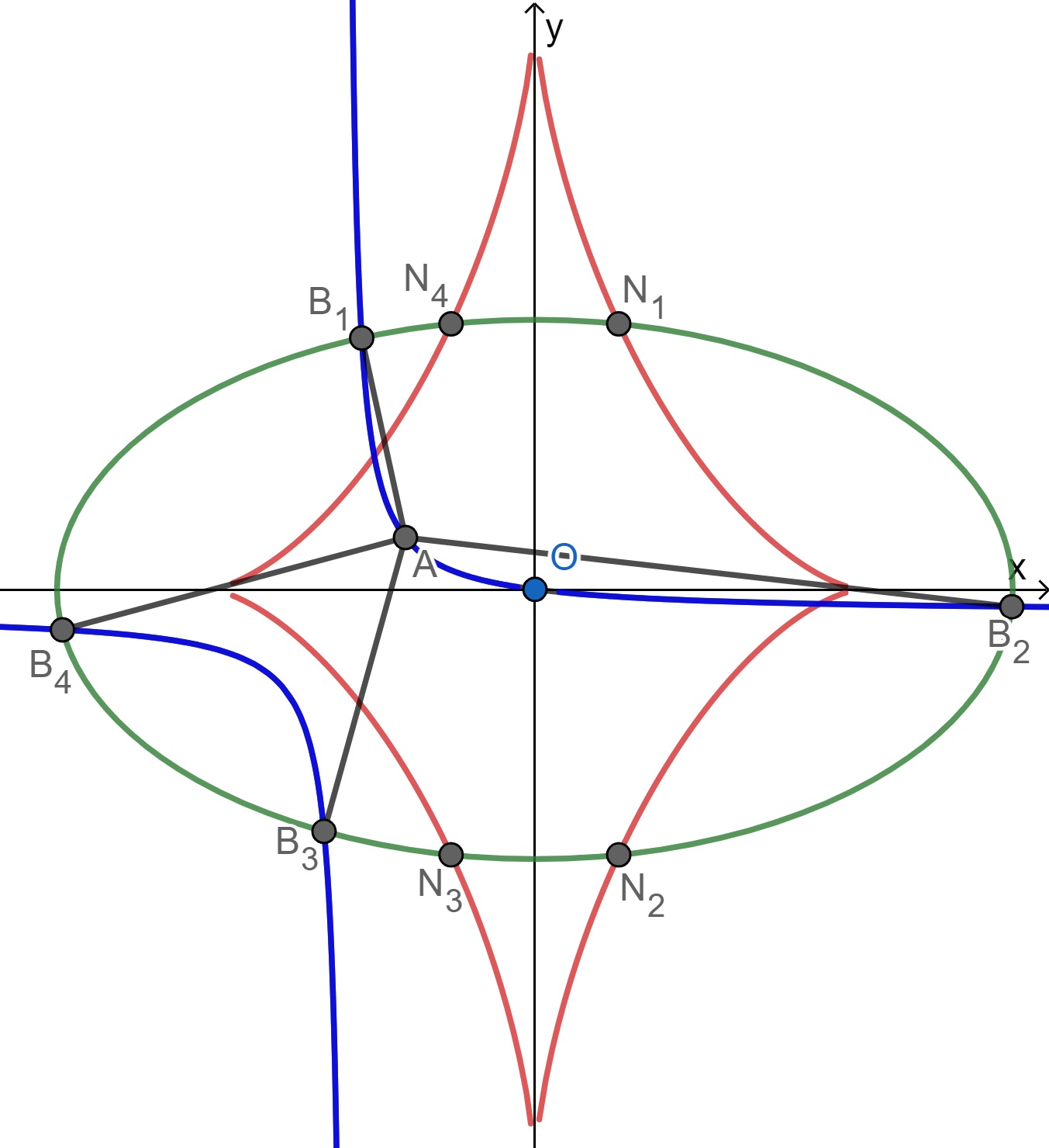}}
\label{fig1}
Figure 1: Ellipse $\frac{x^2}{a^2}+\frac{y^2}{b^2}=1$ (green), its 4 normals (black), Apollonius hyperbola $y=\frac{xY}{\epsilon X-(\epsilon -1)x}$ (blue), and astroida $\sqrt[3]{a^2X^2}+\sqrt[3]{b^2Y^2}=\sqrt[3]{\left(a^2-b^2\right)^2}$ (red). Created using GeoGebra.
\end{figure}

The problem about the number of normals, which Apollonius called as \textit{the shortest} and sometimes \textit{the longest line segments}, appeared in the fifth book of Apollonius, which survived only in Arabic translation \cite{milne}. For the outline of the solution of Apollonius, one can check \cite{van}, Chapter VII, p. 260-261. There is a lively discussion of this problem also in pages 131-135 of \cite{tikh}, \cite{hart}. The problem was also mentioned by V.I. Arnold in his paper \cite{arnold}, Chapter IV and related popular lecture \cite{arnold2} available online both as a brochure and as a YouTube video. The main objective of this paper is to study in detail the cases when the point $A$ is on the ellipse (1), and generalize these results to three dimensions for ellipsoids. Some of the results in the current paper were presented at the Maple Conference 2022 \cite{aliyev}, again available as a YouTube video.

\textbf{2. Apollonius problem for plane}

Let the ellipse be defined by
$$\frac{x^2}{a^2}+\frac{y^2}{b^2}=1, \eqno(1)
$$
where we assume that $a>b>0$. Let us take an arbitrary point $A(X,Y)$ on the plane of the ellipse. We want to find point $B(x,y)$ on the ellipse such that $AB$ is perpendicular to the tangent of the ellipse at $B$. The slope of this tangent line is $y'=-\frac{b^2x}{a^2y}$, and therefore $\frac{y-Y}{x-X}=\frac{a^2y}{b^2x}$. From this we obtain the equation of rectangular hyperbola $y=\frac{xY}{\epsilon X-(\epsilon-1)x},$ where $\epsilon=\frac{a^2}{b^2}$. The intersection points $B_1,\ B_2,\ B_3$, and $B_4$ of this hyperbola with the ellipse give us the required normals $AB_1$, $AB_2$, $AB_3$, and $AB_4$. In his solution, Apollonius also used this hyperbola, which is now known as Apollonius hyperbola \cite{berger}, Sect. 17.5.5.6. The asymptotes of the hyperbola are $x=\frac{a^2X}{a^2-b^2}$ and $y=\frac{b^2Y}{b^2-a^2}$. One of the branches of this hyperbola passes through the center of the ellipse and therefore, there are at least 2 intersection points with the ellipse. The other branch may or may not intersect the ellipse. In the cases when $X=0$ and $Y=0$, the hyperbola degenerates to a pair of perpendicular lines $x=0,y=\frac{b^2Y}{b^2-a^2}$ and $x=\frac{a^2X}{a^2-b^2},y=0$, respectively. Let us denote by $n(A)$ the total number of intersections of the hyperbola with the ellipse. Since the intersection points are the solutions of a fourth order equation, $n(A)$ can not exceed 4. Let us find points $A$, where $n(A)$ jumps from 4 to 2. This happens when Apollonius hyperbola is tangent to the ellipse i.e. the slopes are equal at the intersection point: $-\frac{x}{\epsilon y}=\frac{\epsilon XY}{\left(\epsilon X-(\epsilon-1)x\right)^2}$. Using this and the equation of the ellipse, we obtain $$\frac{x}{a}=\frac{a}{a^2-b^2}\left(\sqrt[3]{\frac{b^2Y^2X}{a^2}}+X\right),\ \frac{y}{b}=\frac{b}{b^2-a^2}\left(\sqrt[3]{\frac{a^2X^2Y}{b^2}}+Y\right),$$ which when used back in the equation of the ellipse, after some simplifications gives
$$
\sqrt[3]{a^2X^2}+\sqrt[3]{b^2Y^2}=\sqrt[3]{\left(a^2-b^2\right)^2}. \eqno(2)
$$
It is the equation of \textit{astroida} in $X,Y$ coordinates. This curve is of 6th order and its parametric equations can be written as $\left(x,y\right)=\left(\frac{a^2-b^2}{a}\cos^3{t},\frac{a^2-b^2}{b}\sin^3{t}\right)$. In the interior region of this astroida $n(A)=4$. Outside of the astroida $n(A)=2$. On the astroida itself $n(A)=3$, except the vertices of the astroida $\left(\pm\frac{a^2-b^2}{a},0\right)$ and $\left(0,\pm\frac{a^2-b^2}{b}\right)$, where again $n(A)=2$. This is essentially what was done by Apollonius, which is a remarkable achievement, taking into account the mathematical tools available at the time. In \cite{berger}, Sect. 17.7.4 (see also p. 204, \cite{nikolsky}) it was mentioned that this astroida is the evolute of the ellipse and therefore drawing normals to the ellipse can be done by drawing tangent lines of the astroida.

Let us now suppose that the point $A(X,Y)$ is on the ellipse: $X=x,\ Y=y$. Since Apollonius hyperbola passes through $A(X,Y)$, one of the points $B_1$, $B_2$, $B_3$, and $B_4$, coincide with $A$. For the points of the ellipse (1) in the astroida (2), $n(A)=4$. For the points of the ellipse (1) outside the astroida (2), $n(A)=2$. For the intersection points $N_1,\ N_2,\ N_3,$ and $N_4$ of the ellipse (1) and the astroida (2), $n(A)=3$.  The coordinates of these points can be easily determined: $(\pm x_0, \pm y_0)$ and $(\pm x_0, \mp y_0)$, where $$x_0=\sqrt{\frac{a^4(a^2-2b^2)^3}{(a^2-b^2)(a^2+b^2)^3}},\ y_0=\sqrt{\frac{b^4(2a^2-b^2)^3}{(a^2-b^2)(a^2+b^2)^3}}.$$ Thus we proved the following

\begin{theorem} For the ellipse (1) and the astroida (2), the following cases are possible:
\begin{enumerate}
\item
If $a^2>2b^2$ then the points $(\pm x_0, \pm y_0)$ and $(\pm x_0, \mp y_0)$ separate the ellipse into 4 regions where $n(A)=4$ and $n(A)=2$. \item If $a^2\le 2b^2$, then for all the points of the ellipse (1), $n(A)=2$.
\end{enumerate}
\end{theorem}
Noting this, we can say that Apollonius problem for the number of concurrent normals of an ellipse is completely solved. There is also a three dimensional variant of this problem, where one takes point $A(X,Y,Z)$ outside of the plane of the ellipse $\frac{x^2}{a^2}+\frac{y^2}{b^2}=1,\ z=0$ and counts the number of lines $AB$, such that $B(x,y,0)$ is on the ellipse, and $AB$ is perpendicular to the tangent of the ellipse at the point $B$. But this variant is easily reduced to the planar case. Consider the projection $A'(X,Y,0)$ of $A$ onto the plane $z=0$. If $A'B$ is a normal of the ellipse then by The Theorem of the Three Perpendiculars, $AB$ is also perpendicular to the tangent of the ellipse at the point $B$. Therefore, $n(A)$ is 2, 3, or 4 dependending on the position of point $A$ with respect cylindrical surface defined by the same equation for the astroida (2).

Apollonius did not mention any practical uses for his results, except that these normals corresponding to minimal and maximal distances, are worth investigating for their own sake and that, in contrast to the tangents (See Appendix), the normals were not studied much by the earlier mathematicians. Because of this connection with the extremal distances, there can be applications in optics, wavefronts, mathematical billiards, etc. One of the applications of these results in astronomy can be a possible explanation for the presence of 4 images of a distant quasar, whose light is being bent around approximately elliptical Einstein Ring formed by two galaxies 3.4 billion light-years away \cite{hubble} (see also Figure 1 in \cite{schmidt}).

\begin{figure}[htbp]
\centerline{\includegraphics[scale=.5]{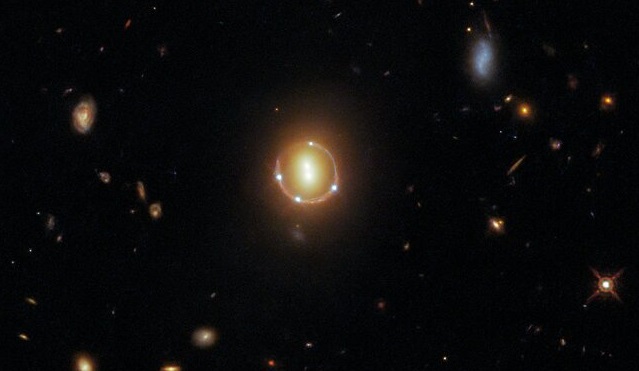}}
\label{fig2}
Figure 2: A possible application of Apollonius Problem. Image credit: ESA/Hubble \& NASA, T. Treu; Acknowledgment: J. Schmidt. Public domain. Used with permission.
\end{figure}

\textbf{3. Apollonius problem for space}

Let us now consider three dimensional generalization of this problem. How many concurrent normals of an ellipsoid are there? In this form the problem was studied through analytic methods in \cite{joach}, \cite{clebsch} (see also \cite{geiser}, \cite{niem}  for geometric considerations) and generalized for higher dimensions in \cite{muller}. The answer to this question will be given in the next section. The literature about the problem of normals to surfaces of second order is vast and we refer the reader to Chapter III, Sections E2 and E3 of \cite{nad}, which contains detailed discussion of the history and many references for this 3D case and the previous planar case.

Let an ellipsoid be defined by
$$
\frac{x^2}{a^2}+\frac{y^2}{b^2}+\frac{z^2}{c^2}=1, \eqno(3)
$$
where we assume that $a>b>c>0$. Let us take an arbitrary point $A(X,Y,Z)$ and find the number $n(A)$ of points $B(x,y,z)$ on the ellipsoid such that $AB$ is the normal line of the plane tangent to the ellipsoid at $B$. Since the outer normal vector of the plane tangent to the ellipsoid at $B(x,y,z)$ is $\textbf{N}=\left(\frac{x}{a^2},\frac{y}{b^2},\frac{z}{c^2}\right)$, $$\frac{x-X}{\frac{x}{a^2}}=\frac{y-Y}{\frac{y}{b^2}}=\frac{z-Z}{\frac{z}{c^2}}=-t,$$
where $t$ is a parameter. From this we find parametric representation of \textit{cubic hyperbola} (see p. 204, \cite{nad})
$$\textbf{r}(t)=\left(\frac{a^2X}{a^2+t},\frac{b^2Y}{b^2+t},\frac{c^2Z}{c^2+t}\right),$$
whose intersections with the ellipsoid give the base points of the normals through $A$. 
The asymptotes of this curve are the lines
$$\textbf{r}_1(t)=\left(t,\frac{b^2Y}{b^2-a^2},\frac{c^2Z}{c^2-a^2}\right),$$
$$\textbf{r}_2(t)=\left(\frac{a^2X}{a^2-b^2},t,\frac{c^2Z}{c^2-b^2}\right),$$
$$\textbf{r}_3(t)=\left(\frac{a^2X}{a^2-c^2},\frac{b^2Y}{b^2-c^2},t\right).$$
If $X=0$, $Y=0$, and $Z=0$ then cubic hyperbola splits into a line, which served earlier as an asymptote of the cubic hyperbola, and a hyperbola:
$$\textbf{r}(t)=\textbf{r}_1(t),\ \textbf{r}(t)=\left(0,\frac{b^2Y}{b^2+t},\frac{c^2Z}{c^2+t}\right);$$
$$\textbf{r}(t)=\textbf{r}_2(t),\ \textbf{r}(t)=\left(\frac{a^2X}{a^2+t},0,\frac{c^2Z}{c^2+t}\right);$$
$$\textbf{r}(t)=\textbf{r}_3(t),\ \textbf{r}(t)=\left(\frac{a^2X}{a^2+t},\frac{b^2Y}{b^2+t},0\right),$$
respectively. The cubic hyperbola passes through the center of the ellipsoid when $t=\pm \infty$, and goes to infinity when $t=-a^2,-b^2,-c^2$. Therefore, there are at least 2 intersections with the ellipsoid. For example, one can take the points of the ellipsoid with maximal and minimal distances from $A$. On the other hand, these intersections are determined by
$$\left(\frac{aX}{a^2+t}\right)^2+\left(\frac{bY}{b^2+t}\right)^2+\left(\frac{cZ}{c^2+t}\right)^2=1, \eqno(4)$$
which is a sixth order equation with respect to $t$, and therefore can not have more than 6 real solutions.
\begin{figure}[htbp]
\centerline{\includegraphics[scale=.9]{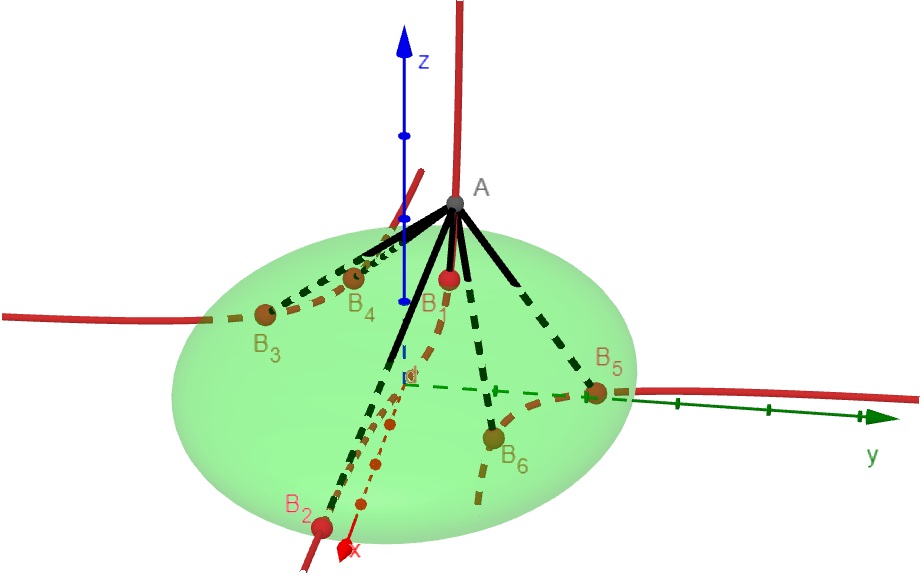}}
\label{fig3}
Figure 3: Ellipsoid $\frac{x^2}{a^2}+\frac{y^2}{b^2}+\frac{z^2}{c^2}=1$ (green), its 6 normals $AB_1, AB_2,\ldots ,AB_6$ (black), cubic hyperbola $\textbf{r}(t)=\left(\frac{a^2X}{a^2+t},\frac{b^2Y}{b^2+t},\frac{c^2Z}{c^2+t}\right)$ (red).  Created using GeoGebra.
\end{figure}
As before, let us denote the number of normals through $A$ by $n(A)$. We want to find points $A$, where $n(A)$ jumps from 2 to 4, or from 4 to 6. This happens when cubic hyperbola is tangent to the ellipsoid i.e. $\textbf{r}'(t)=\left(-\frac{a^2X}{(a^2+t)^2},-\frac{b^2Y}{(b^2+t)^2},-\frac{c^2Z}{(c^2+t)^2}\right),$ is orthogonal to $\textbf{N}=\left(\frac{X}{a^2+t},\frac{Y}{b^2+t},\frac{Z}{c^2+t}\right)$. This can be expressed as $\textbf{r}'(t)\cdot\textbf{N}=0$, or as
$$\frac{a^2X^2}{(a^2+t)^3}+\frac{b^2Y^2}{(b^2+t)^3}+\frac{c^2Z^2}{(c^2+t)^3}=0. \eqno(5)$$
The equations (4) and (5) define the surface known as \textit{Caustics of an Ellipsoid} also known as \textit{focal surface, surface of centers, evolute of an ellipsoid}, or just \textit{Cayley's astroida} \cite{cayley} (see also p. 165, \cite{weath}). Cayley used the name \textit{Centro-surface of an Ellipsoid}, and the equations (4) and (5), which appear in p. 358 of \cite{cayley}, were obtained using the fact that the points of this surface are the centers of principal curvatures of the ellipsoid (3) (see also p. 218 in \cite{salmon}). The centers and principal radii of curvature for more general surfaces were studied by G. Monge in Sect. XXV of \cite{monge}. For a modern definition of the principal curvatures, see, for example, p. 158, \cite{fischer}. A. Cayley's graph of the surface appears in p. 330 of \cite{cayley} (also shown in p.116, \cite{forsyth}). One can also find many other images depicting this surface in various papers, dissertations, and books. See for example pp. 49-53 in  \cite{banc}, p. 154 in Ch. 7 of \cite{banc2}, \cite{dom}, \cite{itoh}, p. 218 in \cite{afraj}, p. 257 in \cite{yang},  p. 49 in \cite{joet} (also shown in p. 356, \cite{berger3}), \cite{yoshi}, p. 61 in \cite{junker}, p. 10 in \cite{cherrie}. The part of the surface where the two surfaces corresponding to minimal and maximal curvatures intersect (named as "the purse"), was shown and mentioned in pages 37 and 109, respectively, of  \cite{arnold3} (see also p. 218 in \cite{afraj}). S.K. Lando gave two popular lectures about the caustics, available online, one with a demonstration of the surface at the end \cite{lando}. Another representation of the surface together with some applications of it in astronomy and physics appeared in \cite{tsar2} (See also \cite{tsar}). According to \cite{tsar2}, the idea of using more general caustics in cosmology is due to Ya. B. Zel'dovich (see  \cite{zeld2}, \cite{zeld} and the references therein).

There are many visualizations of this surface as a physical model. Before the dawn of computer graphics and 3D printers, handmade models and sculptures represented the best medium for such mathematical objects \cite{fischer}. In \cite{kummer}, there is a description of a model made out of gypsum by student H.A. Schwarz in the Arts Faculty (later Prof. in Univ. Berlin), which is also mentioned in Sect. 197 (p. 282), \cite{dyck} (see also p. 198, \cite{nad}). Stereographic photo of one such model by unknown artist/maker from the same time period is shown in Figure 4, \cite{getty}. Two more models of this surface together with models of centro-surfaces of paraboloids and hyperboloids can be found in The Collection of Mathematical Models and Instruments in The University of Göttingen \cite{goet} (See Fig. 5). Similar models for the centers of curvature of paraboloids and hyperboloids were described in \cite{schroder} (see also \cite{caspari}) and p. 283 in \cite{dyck} (see also \cite{dyck2}), respectively (see also p. 264 and p. 34, respectively, in \cite{junker}). See also the website of \textit{The TouchGeometry Project} \cite{space} for models of caustics of an elliptical paraboloid and a hyperboloid of one sheet in Geometry Department of Karazin University in Kharkiv, Ukraine. Another such model is in The National Museum of American History \cite{brill}.
\begin{figure}[htbp]
\centerline{\includegraphics[scale=.3]{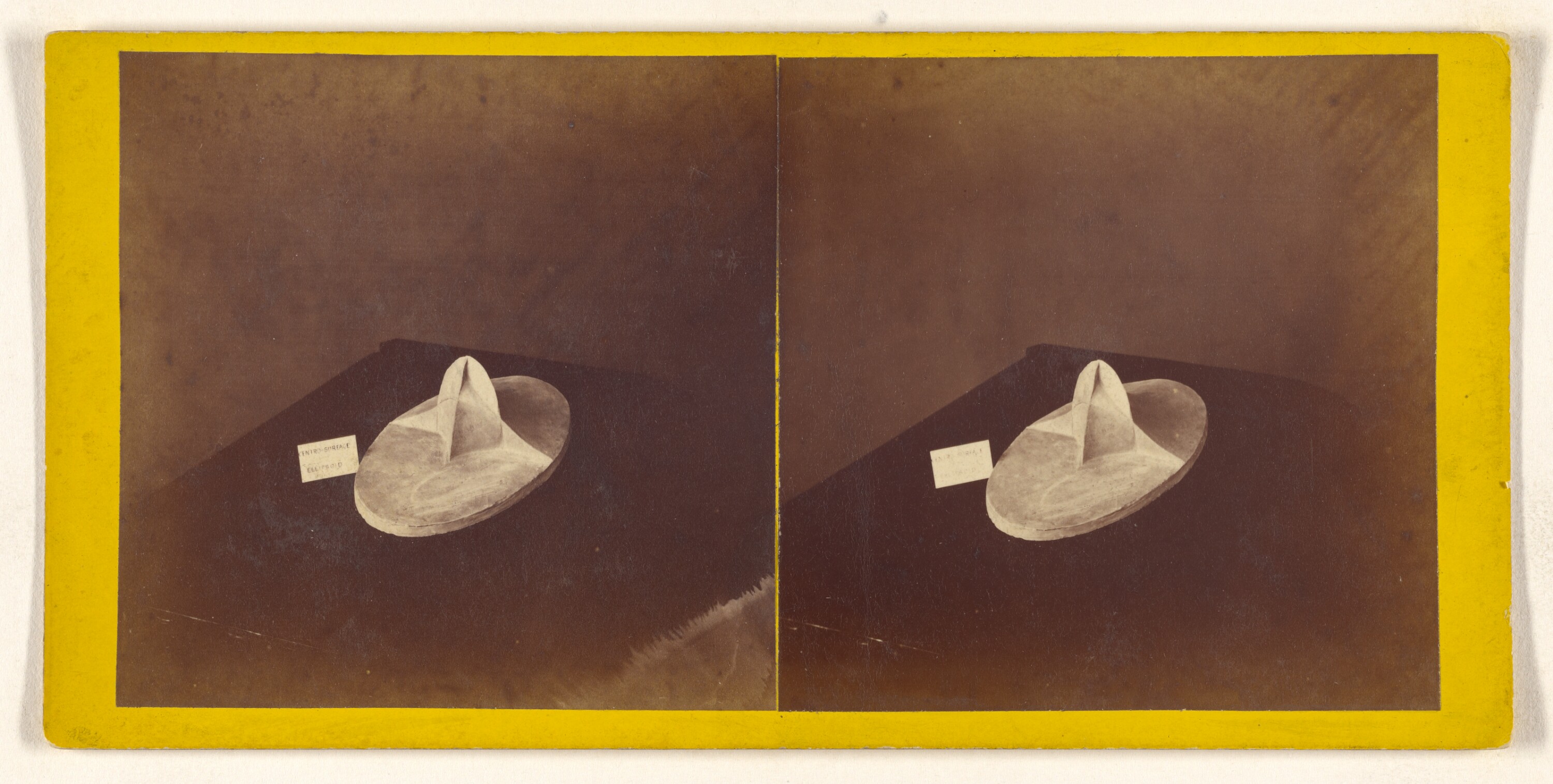}}
\label{fig4}
Figure 4: Stereograph Card, Unknown artist/maker, Centro-Surface. Ellipsoid. about 1860. 
Gift of Weston J. and Mary M. Naef, Getty Museum Collection. Open Content program. No copyright. Used with permission.
\end{figure}

\begin{figure}[htbp]
\centerline{\includegraphics[scale=.2]{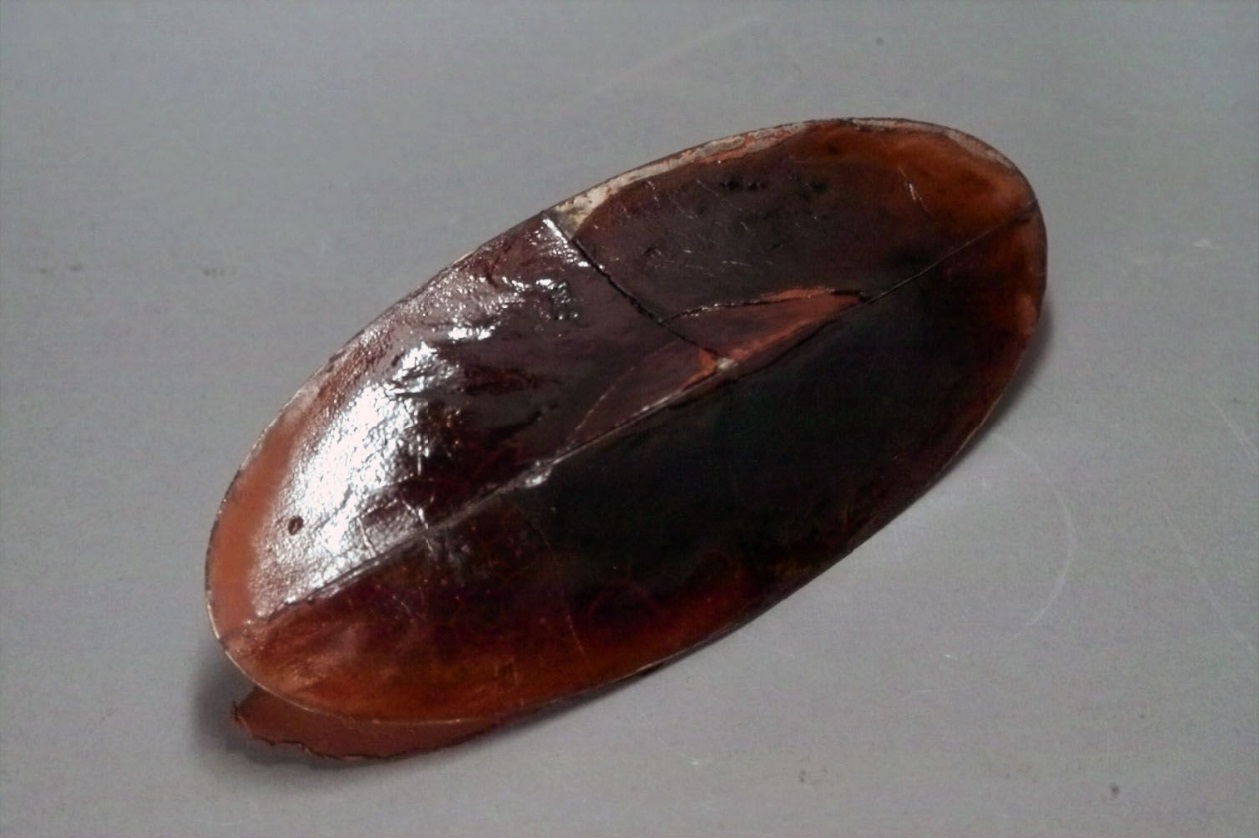}
\includegraphics[scale=.3]{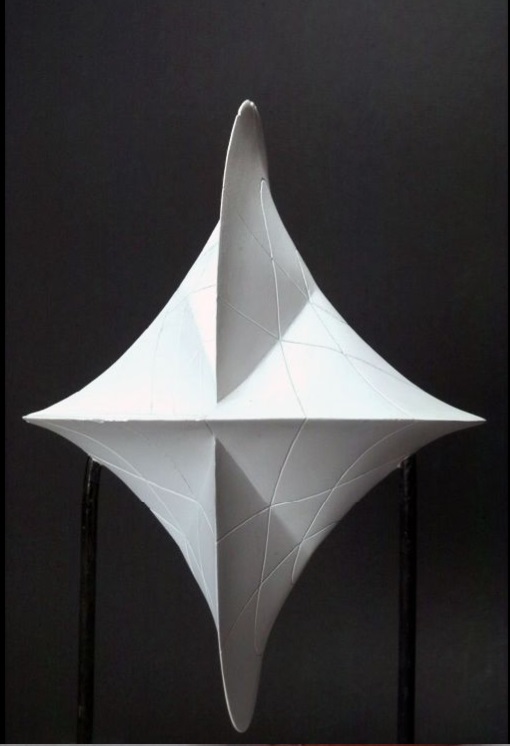}}
\label{fig5}
Figure 5: Curvature centre point surface models: 239 and 242.
Curvature centre point surface of the triaxial ellipsoid. Gypsum; curved surfaces are unified. (Göttinger Sammlung mathematischer Modelle und Instrumente,  Georg-August-Universität Göttingen)
\end{figure}

Note that in general, it is not easy to exclude the parameter $t$ from the equations (4) and (5), to get an explicit equation for the caustics (see \cite{salmon}, p. 113 in \cite{forsyth}). This surface is of 12th order and its equation can be written as a $5\times 5$ determinant (see p. 114 in \cite{forsyth}). But if, for example, $c=0$, then the equations (4) and (5) are transformed to
$$\left(\frac{aX}{a^2+t}\right)^2+\left(\frac{bY}{b^2+t}\right)^2=1,\ \frac{a^2X^2}{(a^2+t)^3}+\frac{b^2Y^2}{(b^2+t)^3}=0,$$
from which one can easily eliminate parameter $t$, and obtain the equation (2) for the astroida. This gives us another solution for the planar case considered in the previous section. Similarly, if $b=c$ then one can introduce a new variable $Y'$, such that $(Y')^2=Y^2+Z^2$ and then the equations (4) and (5) can be written as
$$\left(\frac{aX}{a^2+t}\right)^2+\left(\frac{bY'}{b^2+t}\right)^2=1,\ \frac{a^2X^2}{(a^2+t)^3}+\frac{b^2(Y')^2}{(b^2+t)^3}=0,$$
from which again the parameter $t$ is easily eliminated to get $$
\sqrt[3]{a^2X^2}+\sqrt[3]{b^2(Y')^2}=\sqrt[3]{\left(a^2-b^2\right)^2},
$$
or
$$
\sqrt[3]{a^2X^2}+\sqrt[3]{b^2(Y^2+Z^2)}=\sqrt[3]{\left(a^2-b^2\right)^2},
$$
which is a surface of revolution generated by rotating astroida (2) around $x$ axis (see Fig. 6).
\begin{figure}[htbp]
\centerline{\includegraphics[scale=1.4]{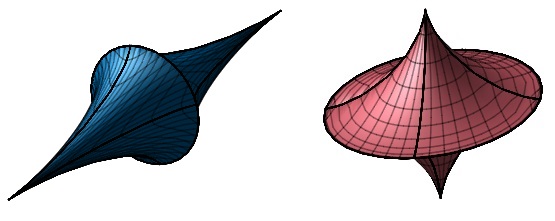}}
\label{fig6}
Figure 6: The surface of revolution generated by rotating the astroida around one of its axes. Created using Maple 2022 for the cases $(a,b,c)= (4,4,3)$ (left) and $(4.1,3,3)$ (right).
\end{figure}

\textbf{4. Caustics of Ellipsoid in GeoGebra and Maple} In this section a method of generating the surface, based on the cartesian coordinates, will be described. The formulas for Gaussian curvature and mean curvature of an ellipsoid are given in Corollary 13.41, p. 413 in \cite{abbena} (see also Chapter 4, \cite{bektas} for applications in geodesy):
$$K(x,y)=\frac{1}{\left(abc\left(\frac{x^2}{a^4}+\frac{y^2}{b^4}+\frac{z^2}{c^4}\right)\right)^2},\ H(x,y)=\frac{|x^2+y^2+z^2-a^2-b^2-c^2|}{2(abc)^2\left(\frac{x^2}{a^4}+\frac{y^2}{b^4}+\frac{z^2}{c^4}\right)^{\frac{3}{2}}}.
$$
The principal curvatures $k_1$ and $k_2$ are the roots of the quadratic equation $x^2-2Hx+K=0$ (Corollary 13.26, p. 400, in \cite{abbena}):
$$
k_1=H-\sqrt{H^2-K},\ k_2=H+\sqrt{H^2-K}.
$$
The corresponding radii of the curvature are $R_1=\frac{1}{k_1}$ and $R_2=\frac{1}{k_2}$, and the respective centers of the curvature $C_1(x_1,y_1,z_1)$ and $C_2(x_1,y_1,z_1)$ can be determined using the formula (see p. 226, \cite{eisen})
$$
C_1(x_1,y_1,z_1)=(x,y,z)-R_1\cdot \frac{\textbf{N}}{|\textbf{N}|},\ C_2(x_2,y_2,z_2)=(x,y,z)-R_2\cdot \frac{\textbf{N}}{|\textbf{N}|},
$$
where as before outer normal is $\textbf{N}=\left(\frac{x}{a^2},\frac{y}{b^2},\frac{z}{c^2}\right)$. The GeoGebra Activity demonstrating the surface, can be found in \href{https://www.geogebra.org/m/fchadfbt}{https://www.geogebra.org}. The Maple Learn document can be found in \href{https://learn.maplesoft.com/d/FPPUAPOJFMJFDQJNBJCQKUNUARLTOFPIDSCGLPMJALNOLNOQJQJGJGBHDFKQIRMOLTFFHTIOPMETGNMLDUFNAJFLDIFPLIFSOTGI}{https://learn.maplesoft.com}. The images created using GeoGebra and Maple 2022 are shown in Figure 7 and Figure 8, respectively.
\begin{figure}[htbp]
\centerline{\includegraphics[scale=1]{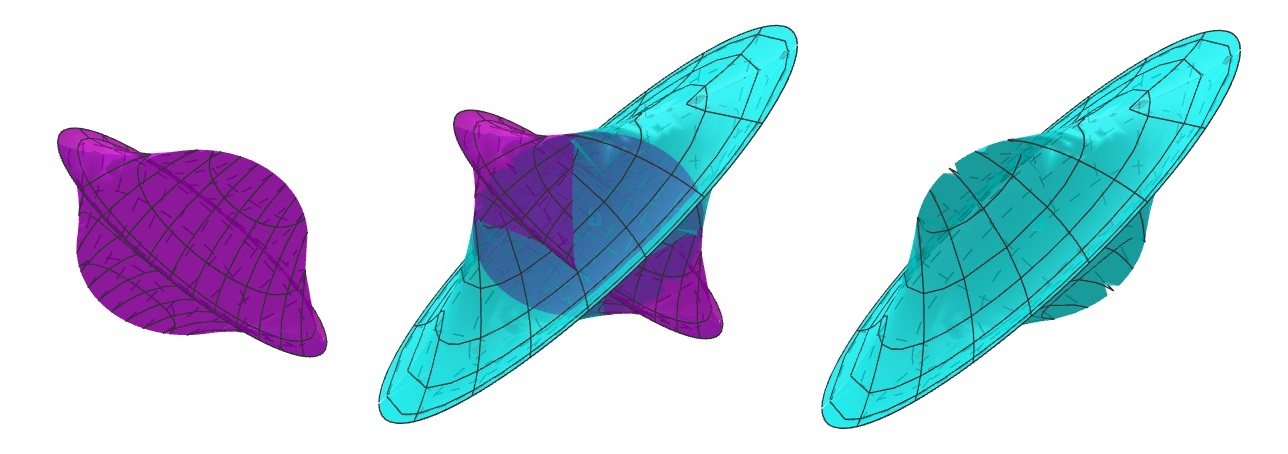}}
\label{fig7}
Figure 7: The centers corresponding to smaller (left, $R_2$) and greater (right, $R_1$) principal radii of curvature. The surfaces intersect (center). Created using GeoGebra.
\end{figure}

\begin{figure}[htbp]
\centerline{\includegraphics[scale=1.4]{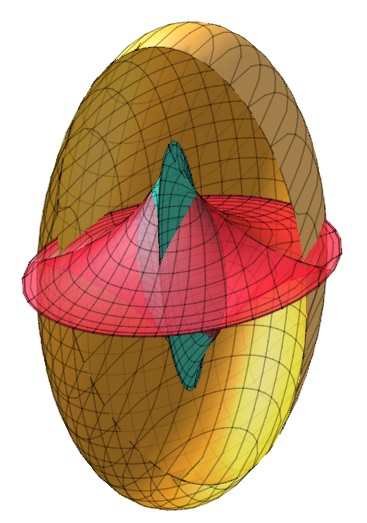} \includegraphics[scale=1.4]{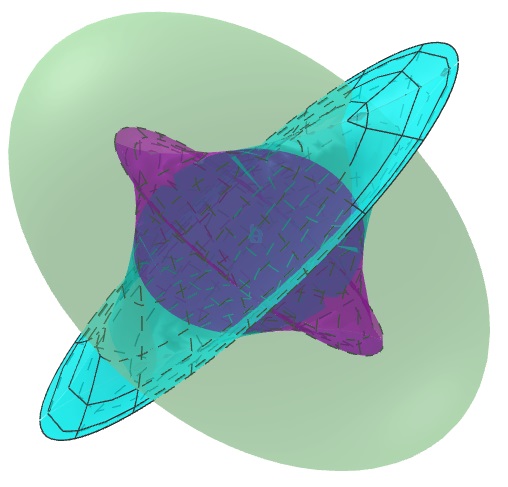}}
\label{fig8}
Figure 8: The ellipsoid and its caustics with transparency applied. Created using Maple 2022 (left) and GeoGebra (right).
\end{figure}

We can now answer the question asked at the beginning of Sect. 2. The number of normals outside of the two caustics is 2 ($n(A)=2$). For the points of the space inside of only one and both of the caustics, $n(A)=4$ and $n(A)=6$, respectively (see \cite{joach}, p. 123-124). On the caustics $n(A)=3$ or $n(A)=5$, with some exceptions on the planes $X=0$, $Y=0$, and $Z=0$ and on the intersections of the two caustics, where again $n(A)=2$ or $n(A)=4$ (see Fig. 9).

In the case of an ellipsoid of revolution, for example, when $a=b$ or $b=c$, one of the caustics becomes a surface of revolution, shown in Figure 6, the other caustic degenerates to a line segment on the axis of symmetry of the surfaces shown in Figure 6, between the vertices. The number of normals of the ellipsoid (3) for the points $A$ on this line segment is infinite ($n(A)=\infty$), except the endpoints of this line segment where $n(A)=2$. For the other points of the space, the situation is identical to the planar case considered in Sect. 2.

\begin{figure}[htbp]
\centerline{\includegraphics[scale=1.3]{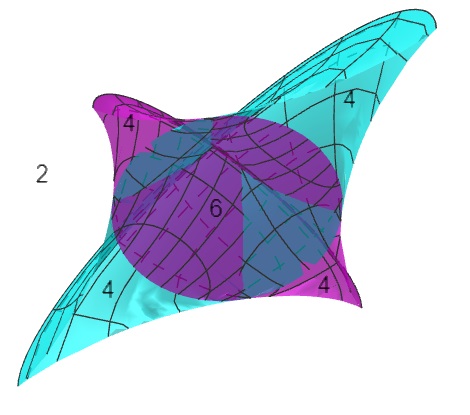}}
\label{fig9}
Figure 9: The number of normals in the regions of space separated by the caustics of the ellipsoid. Half of the caustics is hidden to make the inner regions visible.
\end{figure}

\textbf{5. The intersections of an ellipsoid and its caustics with the coordinate planes}
 
The intersection curves of the ellipsoid (3) and its caustics (see Figure 10) with the coordinate planes $x=0,\ y=0,\ z=0$ are found in p. 325, \cite{cayley}
(see also p. 115, \cite{forsyth}).

\begin{lemma} The intersections of the ellipsoid (3) and its caustics with the coordinate planes are the following curves:
\begin{enumerate}
\item Ellipse $\left(a\cos{t},b\sin{t},0\right)$ (black),
\item Ellipse $\left(a\cos{t},0,c\sin{t}\right)$ (yellow),
\item Ellipse $\left(0,b\cos{t},c\sin{t}\right)$ (red),
\item Astroida $\left(\frac{a^2-b^2}{a}\cos^3{t},\frac{a^2-b^2}{b}\sin^3{t},0\right)$ (pink),
\item Astroida $\left(\frac{a^2-c^2}{a}\cos^3{t},0,\frac{a^2-c^2}{c}\sin^3{t}\right)$ (light blue),
\item Astroida $\left(0,\frac{b^2-c^2}{b}\cos^3{t},\frac{b^2-c^2}{c}\sin^3{t}\right)$ (purple),
\item Ellipse $\left(\frac{a^2-c^2}{a}\cos{t},\frac{b^2-c^2}{b}\sin{t},0\right)$ (green),
\item Ellipse $\left(\frac{a^2-b^2}{a}\cos{t},0,\frac{b^2-c^2}{c}\sin{t}\right)$ (dark blue),
\item Ellipse $\left(0,\frac{a^2-b^2}{b}\cos{t},\frac{a^2-c^2}{c}\sin{t}\right)$ (orange).
\end{enumerate}
\end{lemma}

\begin{figure}[htbp]
\centerline{\includegraphics[scale=1.3]{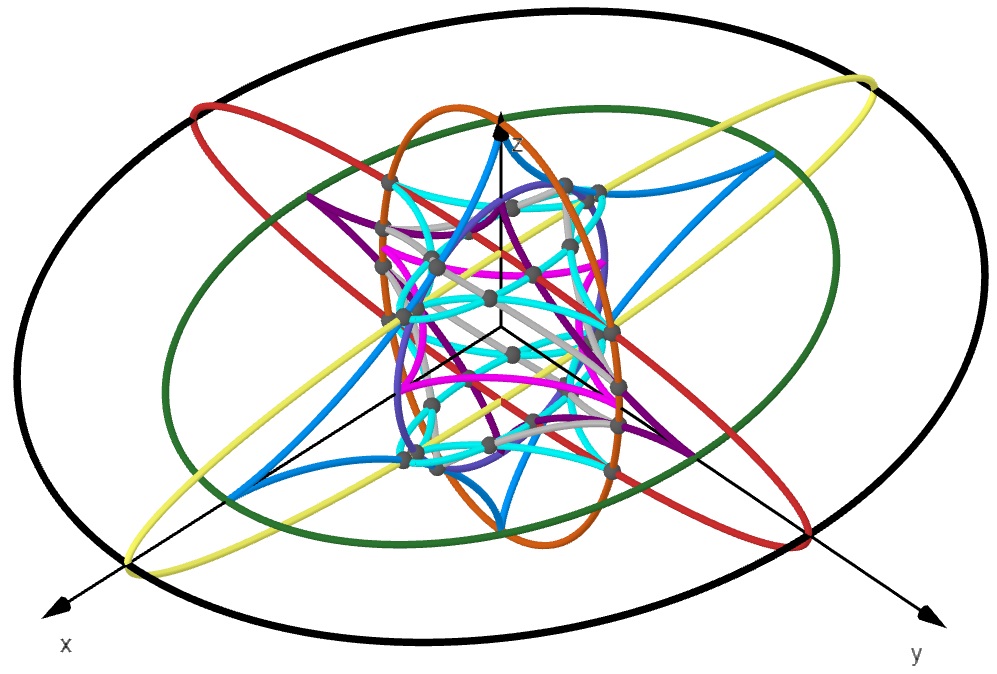}}
\label{fig10}
Figure 10: The intersections of a triaxial ellipsoid and its caustics with the coordinate planes, the nodal curve (grey), and the intersections of the ellipsoid with its caustics (cyan blue). The other colors are explained in Lemma 2. The tangency points of the two caustics and some of the intersection points are also shown. See \url{https://www.geogebra.org/3d/tqjwgxwg} for more details.
\end{figure}

Let us now find intersections of Ellipses 1,2,3, Astroidas 4,5,6 and Ellipses 7,8,9, respectively.
\begin{lemma}

\begin{enumerate}
\item If $a^2\ge 2b^2$, then Ellipse 1 and Astroida 4 intersect at $(\pm x_0, \pm y_0,0)$ and 
$(\pm x_0, \mp y_0,0)$, where $$x_0=\sqrt{\frac{a^4(a^2-2b^2)^3}{(a^2-b^2)(a^2+b^2)^3}},\ y_0=\sqrt{\frac{b^4(2a^2-b^2)^3}{(a^2-b^2)(a^2+b^2)^3}}.$$
\item If $a^2\ge 2c^2$, then Ellipse 2 and Astroida 5 intersect at $(\pm x_1,0, \pm z_1)$ and 
$(\pm x_1,0, \mp z_1)$, where $$x_1=\sqrt{\frac{a^4(a^2-2c^2)^3}{(a^2-c^2)(a^2+c^2)^3}},\ z_1=\sqrt{\frac{c^4(2a^2-c^2)^3}{(a^2-c^2)(a^2+c^2)^3}}.$$
\item If $b^2\ge 2c^2$, then Ellipse 3 and Astroida 6 intersect at $(0,\pm y_2, \pm z_2)$ and 
$(0,\pm y_2, \mp z_2)$, where $$y_2=\sqrt{\frac{b^4(b^2-2c^2)^3}{(b^2-c^2)(b^2+c^2)^3}},\ z_2=\sqrt{\frac{c^4(2b^2-c^2)^3}{(b^2-c^2)(b^2+c^2)^3}}.$$

\item Ellipse 1 and Ellipse 7 do not have real intersection points and the coordinates of the non-real intersection points are $(\pm x^*,\pm y^*,0)$ and 
$(\pm x^*,\mp y^*,0)$, where $$x^*=\sqrt{\frac{a^2(a^2-c^2)^2(2b^2-c^2)}{(a^2-b^2)(2a^2b^2-a^2c^2-b^2c^2)}},$$
$$y^*=\sqrt{\frac{b^2(b^2-c^2)^2(2a^2-c^2)}{(b^2-a^2)(2a^2b^2-b^2c^2-a^2c^2)}}.$$

\item If $b^2\ge 2c^2$, then Ellipse 2 and Ellipse 8 intersect at $(\pm x_3,0, \pm z_3)$ and 
$(\pm x_3,0, \mp z_3)$, where $$x_3=\sqrt{\frac{a^2(a^2-b^2)^2(2c^2-b^2)}{(a^2-c^2)(2a^2c^2-a^2b^2-b^2c^2)}},$$
$$z_3=\sqrt{\frac{c^2(c^2-b^2)^2(2a^2-b^2)}{(c^2-a^2)(2a^2c^2-a^2b^2-b^2c^2)}},$$
and $x_1\ge x_3$, $z_1\le z_3$ with equality cases when $\frac{1}{a^2}+\frac{1}{c^2}=\frac{3}{b^2}$.

\item If $2b^2\ge a^2\ge 2c^2$, then Ellipse 3 and Ellipse 9 intersect at $(0,\pm y_4, \pm z_4)$ and 
$(0,\pm y_4, \mp z_4)$, where $$y_4=\sqrt{\frac{b^2(b^2-a^2)^2(2c^2-a^2)}{(b^2-c^2)(2b^2c^2-a^2b^2-a^2c^2)}},$$
$$z_4=\sqrt{\frac{c^2(c^2-a^2)^2(2b^2-a^2)}{(c^2-b^2)(2b^2c^2-a^2b^2-a^2c^2)}}.$$
In particular, $y_2\ge y_4$ and $z_2\le z_4$ if and only if $2 b^4+2 c^4-a^2 b^2-a^2 c^2-2 b^2 c^2\ge 0$.

\item If $a^2+c^2\ge 2b^2$, then Astroida 4 and Ellipse 7 intersect at $(\pm x_5, \pm y_5,0)$ and 
$(\pm x_5, \mp y_5,0)$, where $$x_5=\sqrt{\frac{(a^2-c^2)^3(2b^2-a^2-c^2)^3}{a^2(b^2-a^2)(a^2+b^2-2c^2)^3}},\ y_5=\sqrt{\frac{(b^2-c^2)^3(2a^2-b^2-c^2)^3}{b^2(a^2-b^2)(a^2+b^2-2c^2)}}.$$
\item Astroida 5 and Ellipse 8 are tangent to each other at the points $(\pm x_6,0, \pm z_6)$ and 
$(\pm x_6,0, \mp z_6)$, where
$$x_6=\sqrt{\frac{(b^2-c^2)^3}{c^2(a^2-c^2)}},\ z_6=\sqrt{\frac{(a^2-b^2)^3}{a^2(a^2-c^2)}},$$
These points also divide Astroida 5 and Ellipse 8 into parts which belong to different caustics. These points are on, in and outside the ellipsoid (3) if $\frac{1}{a^2}+\frac{1}{c^2}=\frac{3}{b^2}$, $<\frac{3}{b^2}$, and $>\frac{3}{b^2}$, respectively.
\item If $a^2+c^2\le 2b^2$, then Astroida 6 and Ellipse 9 intersect at $(0,\pm y_7, \pm z_7)$ and 
$(0,\pm y_7, \mp z_7)$, where $$y_7=\sqrt{\frac{(b^2-a^2)^3(2c^2-b^2-a^2)^3}{b^2(c^2-b^2)(b^2+c^2-2a^2)^3}},$$
$$z_7=\sqrt{\frac{(c^2-a^2)^3(2b^2-c^2-a^2)^3}{c^2(b^2-c^2)(b^2+c^2-2a^2)^3}}.$$
These points are on, in and outside the ellipsoid (3) if $2 b^4+2 c^4-a^2 b^2-a^2 c^2-2 b^2 c^2=0$, $<0$, and $>0$, respectively.
\end{enumerate}
\end{lemma}

\begin{proof} 
By direct substitution of the given coordinates in the equations of the curves, we can  prove the claim about the intersections. The inequalities involving coordinates are also proved by direct substitutions. For tangency of the curves in part 8, additionally, the slopes of the curves are compared. In parts 8 and 9, the last claim about the position with respect to the ellipsoid (3), follows directly from the substitution in (3) of the coordinates $x_6, z_6$ and $y_7, z_7$, respectively.
\end{proof}

One can experiment with these intersection points and the curves by moving the sliders in GeoGebra Activity \url{https://www.geogebra.org/3d/tqjwgxwg}.

\textbf{6. The number of normals for the points of an ellipsoid} As for an ellipse, if the point $A$ is on the ellipsoid (3), then one of the points $B_1, B_2,\ldots,B_6$ coincides with $A$. The description all possible cases for the regions on an ellipsoid, where $n(A)$ jumps from 2 to 4, or from 4 to 6, is not as trivial as in the planar case. It seems that the problem in this setting did not attract much attention and remains unstudied. In the remaining part of the paper we will highlight the 3 cases of intersections of the ellipsoid and its caustics (see Figure 8). In general, these intersections are some curves on the ellipsoid (3), and a simple parametrization for these curves is given at the end of this paper. Using the intersections of these curves with the coordinate planes, which we found in the previous section, one can categorize 3 possible cases: (i) none of the caustics intersect the ellipsoid, (ii) only one of the caustics intersects the ellipsoid, (iii) both of the caustics intersect the ellipsoid. These intersections have many different shapes and positions, and the complete categorization of all general cases is shown in Figure 11.
\begin{theorem} For the ellipsoid (3) and its caustics defined by (4) and (5), the following cases are possible:
\begin{enumerate}
\item If $a^2<2c^2$ then there are no intersections of the caustics with the ellipsoid (3),

\item If $ b^2<2c^2\le a^2$ then only one of the caustics intersects the ellipsoid (3),

\item If $b^2\ge 2c^2$, then both of the caustics intersects the ellipsoid (3).
\end{enumerate}
In all cases, for the points of the ellipsoid (3) lying outside of the two caustics $n(A)=2$, for the points of the ellipsoid (3) lying in only one of these caustics $n(A)=4$, for the points of the ellipsoid (3) lying in both of these caustics $n(A)=6$, and for the intersection points of the ellipsoid (3) and these caustics $n(A)=3$ or $5$, except some of the points of the ellipsoid (3), where the caustics intersect each other or these caustics intersect the coordinate planes.
\end{theorem}

\begin{proof} 
Let us first note that if the ellipsoid (3) and one of its caustics intersect, then they should also intersect on at least one point of the coordinate planes. Indeed, suppose on the contrary that the ellipsoid (3) and one of its caustics intersect but they do not intersect on any of the coordinate planes. Then one of the intersection curves should be situated completely in the 1st octant. Denote this curve by $\Gamma $. The part of the caustic bounded by $\Gamma $ is a smooth surface outside of the ellipsoid (3), and therefore it has a tangent plane which does not intersect the ellipsoid (3). This tangent plane contains also the corresponding normal of the ellipsoid (3) because the principal radii are tangent to the caustics (see p. 312, \cite{kagan}). This is a contradiction because the normals intersects the ellipsoid (3). It follows that the problem of existence of intersections of the ellipsoid (3) and its caustics can be studied just by their cross sections with the coordinate planes, which was done in Lemma 2 and Lemma 3. The remaining claims follow directly from the results in \cite{joach}.
\end{proof}

Depending on whether $a^2\le 2b^2$ or $a^2> 2b^2$, the caustic corresponding to the greater principal radius (the red caustics in Figure 11) is encompassed by the ellipsoid (3) or the ellipsoid (3) is encompassed by this caustic. Similarly, depending on whether $a^2+c^2\ge 2b^2$ or $a^2+c^2<2b^2$, the (blue) caustic corresponding to the smaller principal radius is encompassed by the (red) caustic corresponding to the greater principal radius or vice versa (cf. p. 326 and p. 363, \cite{cayley}). It is obvious that if $a^2+c^2\le 2b^2$ then $a^2<2b^2$. Similarly, if $b^2<2c^2\le a^2$ then $\frac{1}{a^2}+\frac{1}{c^2}<\frac{3}{b^2}$.

Detailed classification of the cases of intersection of the ellipsoid and its caustics is done in Figure 11 based on the sign of the expressions $\frac{1}{a^2}+\frac{1}{c^2}-\frac{3}{b^2}$ and $2 b^4+2 c^4-a^2 b^2-a^2 c^2-2 b^2 c^2$. Note that if $b^2\ge 2c^2$ and $2 b^4+2 c^4-a^2 b^2-a^2 c^2-2 b^2 c^2\ge 0$ then $\frac{1}{a^2}+\frac{1}{c^2}>\frac{3}{b^2}$. Indeed, since $\frac{1}{a^2}+\frac{1}{c^2}>\frac{3}{b^2}$ can be written as $b^2c^2+a^2b^2-3a^2c^2>0$, it is sufficient to show that
$$
b^2c^2+a^2b^2-3a^2c^2>2 b^4+2 c^4-a^2 b^2-a^2 c^2-2 b^2 c^2.
$$
This inequality can be written as $3b^2c^2+2a^2(b^2-c^2)>2 b^4+2 c^4$. Since
$a>b>c$, it is sufficient to show that $3b^2c^2+2b^2(b^2-c^2)\ge2 b^4+2 c^4$, which simplifies to $b^2\ge 2c^2$. Similarly, if $2 b^4+2 c^4-a^2 b^2-a^2 c^2-2 b^2 c^2\ge 0$ then $a^2+c^2<2b^2$. Indeed, by rewriting the given inequality we obtain $a^2+c^2\le 2b^2-\frac{c^2(a^2+b^2-2c^2)}{b^2}<2b^2$. 
\begin{theorem}
If $2 b^4+2 c^4-a^2 b^2-a^2 c^2-2 b^2 c^2\le 0$ and $\frac{1}{a^2}+\frac{1}{c^2}\ge\frac{3}{b^2}$ then the two caustics and the ellipse (3) intersect at a unique point of each octant.
\end{theorem}

\begin{proof} 
The intersection curve of the two caustics (see Figure 10) is called \textit{nodal curve} and its parametrization was given in p. 351, \cite{cayley}:
$$
(x(t))^2=\frac{((\gamma-\alpha)t+\alpha) ((\gamma-\alpha)t-2\gamma)^2 ((\beta-\gamma)t+\gamma)^3}{-\beta\gamma a^2(\alpha \gamma+\Omega t)(3 t-2)^2},
(y(t))^2=\frac{t^2(t-1)((\gamma-\alpha)^2t+3\alpha\gamma)^3}{-\alpha\gamma b^2(\alpha \gamma+\Omega t)(3 t-2)^2},
$$
$$
(z(t))^2=\frac{((\gamma-\alpha)t-\gamma) ((\gamma-\alpha)t+2\alpha)^2 ((\alpha-\beta)t-\alpha)^3}{-\alpha\beta c^2(\alpha \gamma+\Omega t)(3 t-2)^2},
$$
where $\alpha=b^2-c^2$, $\beta=c^2-a^2$, $\gamma=a^2-b^2$, $\Omega=\alpha^2-\beta \gamma$, and $0\le t\le \min\left(\frac{\alpha}{\alpha-\beta},-\frac{\gamma}{\beta-\gamma}\right)$. Note that $-\frac{\gamma}{\beta-\gamma}\ge\frac{\alpha}{\alpha-\beta}$ iff $a^2+c^2\ge 2b^2$. Also note that $\max\left(\frac{\alpha}{\alpha-\beta},-\frac{\gamma}{\beta-\gamma}\right)<\frac{2}{3}$. Consider the function $f(t)=\frac{x(t)^2}{a^2} + \frac{y(t)^2}{b^2} + \frac{z(t)^2}{c^2}-1$. Simplifying using Maple we obtain that $f(t)=\frac{p(t)(q(t))^2}{a^4b^4c^4\left(3t-2\right)^2r(t)}$, where
$$
p(t)=(a^4b^2 + a^4c^2 + a^2b^4  + c^4a^2 + b^4c^2 +c^4 b^2- 6a^2b^2c^2)(a^2 - 2b^2 + c^2)t^2 
$$
$$
+(11a^4b^2c^2  + 11a^2b^2c^4  + 3a^2b^6 + 3b^6c^2  -b^2a^6 - c^2a^6 - a^2c^6 - b^2c^6- a^4b^4 - 5a^4c^4- b^4c^4- 17a^2b^4c^2)t
$$
$$
+ (a^2 - b^2)(b ^2- c^2)(a^2b^2 + b^2c^2 - 3a^2c^2),
$$
$$
q(t)=(a^2 + b^2 + c^2)(a^2 - 2b^2 + c^2)t^2 + (4b^4 + a^2b^2 + b^2c^2 - 3a^2c^2)t - 2b^4,
$$
$$
r(t)=(a^4 +b^4 + c^4-a^2b^2 - b^2c^2 - a^2c^2 )t + (a^2-b^2)(b^2-c^2).
$$
Note that $r(t)>0$ whenever $t\ge0$. We will consider 3 cases.

Case 1. Suppose that $a^2 - 2b^2 + c^2>0$. Note that
$$
p(0)=(a^2 - b^2)(b^2 - c^2)(a^2b^2 + b^2c^2 - 3a^2c^2)>0,
$$
$$
p\left(\frac{\alpha}{\alpha-\beta}\right)=-\frac{c^2(a^2 - c^2)(b^2 - c^2)(2a^2 - b^2- c^2)( 2a^4+ 2b^4-a^2c^2 - b^2c^2  - 2a^2b^2 )}{(a^2 + b^2 - 2c^2)^2}<0
$$
$$
p\left(1\right)=-b^4(a^2 - c^2)^2<0, \ p(+\infty)=+\infty.
$$
Therefore, one of the roots of $p(t)$ is in the interval $\left(0,\frac{\alpha}{\alpha-\beta}\right)$, and the other root is in the interval $(1,+\infty)$.
Similarly,
$$
q(-\infty)=+\infty,\ q(0)=-2b^4<0,
$$
$$
q\left(\frac{\alpha}{\alpha-\beta}\right)=-\frac{c^2(3a^2(b^2-c^2)+c^2(a^2+c^2 - 2b^2))(2a^2 - b^2 - c^2)}{(a^2 + b^2 - 2c^2)^2}<0, \ q(+\infty)=+\infty,
$$
Therefore, one of the roots of $q(t)$ is in the interval $(-\infty,0)$, and the other root is in the interval $\left(\frac{\alpha}{\alpha-\beta},+\infty\right)$.

Case 2. Now suppose that $a^2 - 2b^2 + c^2<0$.
$$
p(-\infty)=-\infty,\ p(0)>0, 
$$
$$
p\left(-\frac{\gamma}{\beta-\gamma}\right)=\frac{a^2(a^2 - c^2)(a^2 - b^2)(a^2 + b^2- 2c^2)( 2b^4+ 2c^4-a^2b^2 - a^2c^2  - 2b^2c^2 )}{(2a^2 -b^2 - c^2)^2}<0.
$$
Therefore, one of the roots of $p(t)$ is in the interval $(-\infty,0)$, and the other root is in the interval $\left(0,-\frac{\gamma}{\beta-\gamma}\right)$.
Similarly,
$$
q\left(-\frac{\gamma}{\beta-\gamma}\right)=\frac{a^2(3c^2(b^2-a^2)+a^2(a^2+c^2-2b^2))(a^2 + b^2 - 2c^2)}{(2a^2 - b^2 - c^2)^2}<0,
$$
$$
q(1)=a^4 - a^2c^2 + c^4>0,\ q(+\infty)=-\infty.
$$
Therefore, one of the roots of $q(t)$ is in the interval $\left(-\frac{\gamma}{\beta-\gamma},1\right)$, and the other root is in the interval $(1,+\infty)$.

Case 3. Let us now suppose that $a^2 - 2b^2 + c^2=0$. In this case $p(t)$ and $q(t)$ are linear functions and $-\frac{\gamma}{\beta-\gamma}=\frac{\alpha}{\alpha-\beta}$. As in Case 1, $p(0)>0,$ $p\left(\frac{\alpha}{\alpha-\beta}\right)<0$, and therefore, the only root of $p(t)$ is in the interval $\left(0,\frac{\alpha}{\alpha-\beta}\right)$. As in Case 2, $
q\left(-\frac{\gamma}{\beta-\gamma}\right)<0,$ $q(1)>0$, and therefore, the only root of $q(t)$ is in the interval $\left(-\frac{\gamma}{\beta-\gamma},1\right)$. 

We proved that in all the cases there is only one zero $t_0$ of $f(t)$ in the interval $0\le t\le \min\left(\frac{\alpha}{\alpha-\beta},-\frac{\gamma}{\beta-\gamma}\right)$ and it is one of the solutions of the quadratic equation $p(t)=0$. The $x,\ y,\ z$ coordinates of the intersection point of the two caustics and the ellipsoid (3) in the first quadrant are $x(t_0),\ y(t_0),\ z(t_0)$, respectively (See Figure 10, the intersection point of cyan blue curves).

\end{proof} 

\begin{figure}[htbp]
\centerline{(i)\includegraphics[scale=0.45]{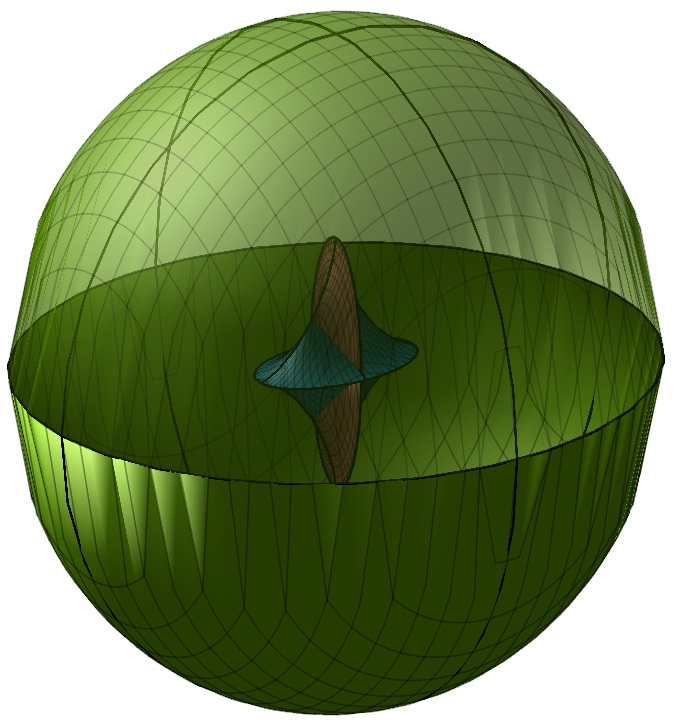} (ii)\includegraphics[scale=0.45]{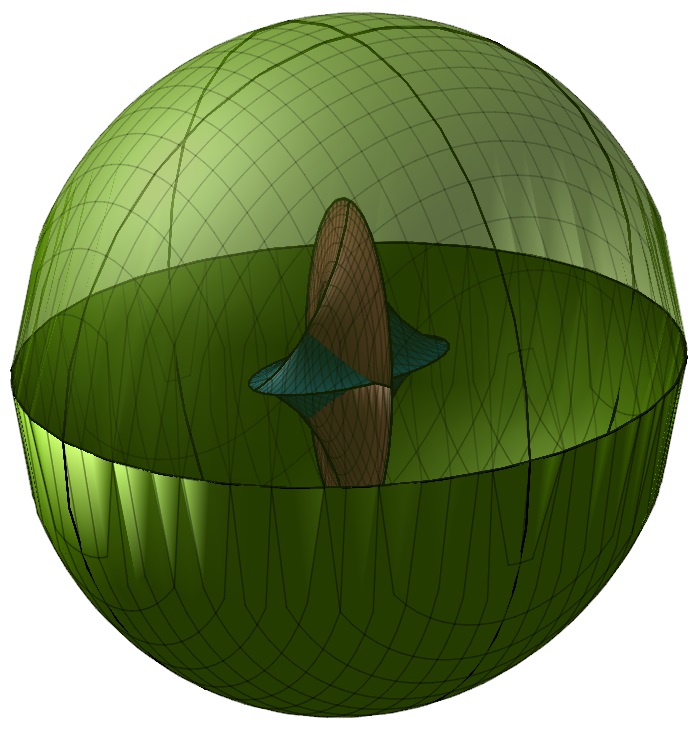} (iii)\includegraphics[scale=0.6]{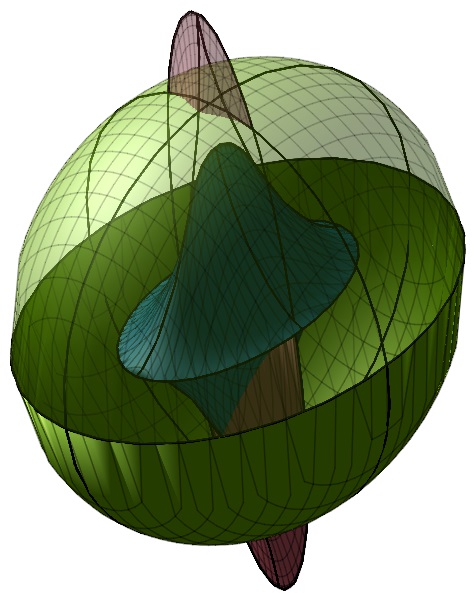} (iv)\includegraphics[scale=0.5]{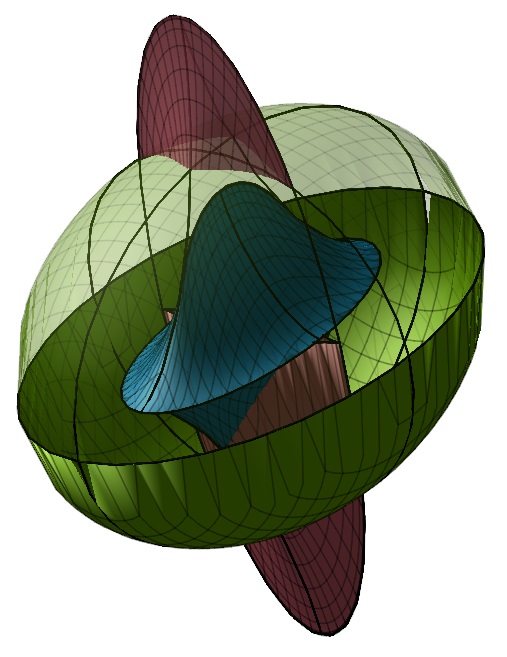}}
\centerline{(v)\includegraphics[scale=0.6]{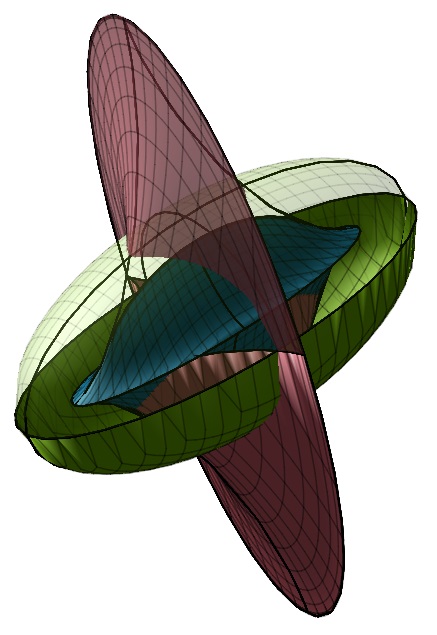}
 (vi)\includegraphics[scale=1]{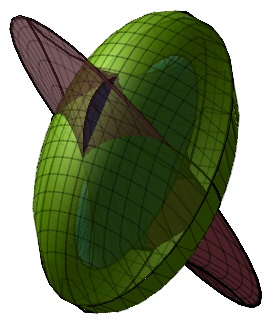} (vii)\includegraphics[scale=0.7]{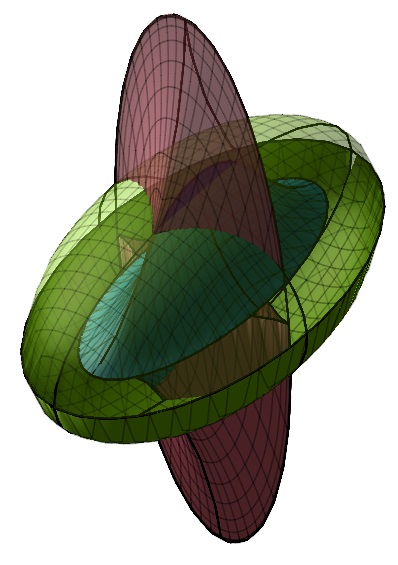} (viii)\includegraphics[scale=0.8]{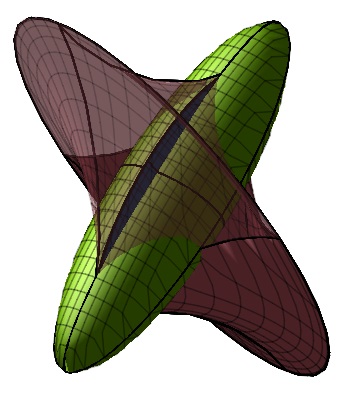}
} 
\centerline{(ix)\includegraphics[scale=0.9]{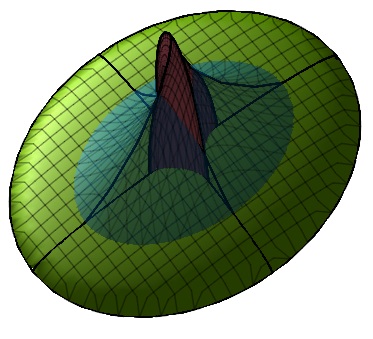} (x)\includegraphics[scale=1]{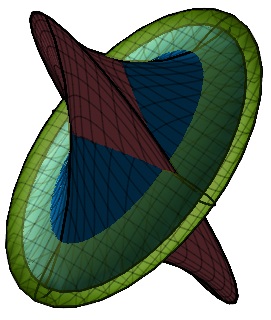} (xi)\includegraphics[scale=0.9]{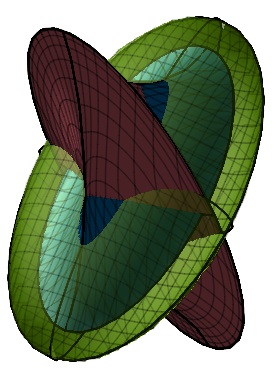}(xii)\includegraphics[scale=0.9]{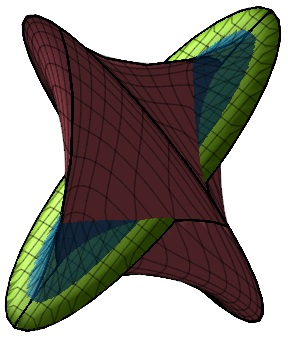}}\label{fig11}

Figure 11: Classification of intersections of ellipsoid (green) and its caustics (red and blue). Created using Maple 2022 for the following values of $(a,b,c)$:

(i) $(4.7,4.4,4)$ ($a^2<2c^2$, $a^2+c^2< 2b^2$).

(ii) $(4.9,4.4,4)$ ($a^2<2c^2$, $a^2+c^2> 2b^2$);

(iii) $(4.7,4,3)$ ($b^2<2c^2< a^2$, $a^2+c^2<2b^2$);

(iv) $(5,4,3)$ ($b^2<2c^2< a^2$, $a^2+c^2> 2b^2$, $a^2<2b^2$);

(v) $(5,3.3,2.5)$ ($b^2<2c^2< a^2$, $a^2> 2b^2$);

(vi) $(4,3,2)$ ($b^2> 2c^2$, $a^2+c^2<2b^2$, $\frac{1}{a^2}+\frac{1}{c^2}<\frac{3}{b^2}$);

(vii) $5, 3.7, 2.5$ ($b^2> 2c^2$, $a^2+c^2> 2b^2$, $a^2<2b^2$, $\frac{1}{a^2}+\frac{1}{c^2}<\frac{3}{b^2}$)

(viii) $(5,2.8,1.8)$ ($b^2> 2c^2$, $a^2> 2b^2$, $\frac{1}{a^2}+\frac{1}{c^2}<\frac{3}{b^2}$);

(ix) $(5,4.4,1.6)$ ($b^2> 2c^2$,  $2 b^4+2 c^4-a^2 b^2-a^2 c^2-2 b^2 c^2>0$); 

(x) $(4.5,3.5,1.4)$ ($b^2> 2c^2$, $a^2+c^2<2b^2$, $\frac{1}{a^2}+\frac{1}{c^2}>\frac{3}{b^2}$,  $2 b^4+2 c^4-a^2 b^2-a^2 c^2-2 b^2 c^2<0$);

(xi) $(5,3.7,2)$  ($b^2> 2c^2$, $a^2+c^2> 2b^2$, $a^2<2b^2$, $\frac{1}{a^2}+\frac{1}{c^2}>\frac{3}{b^2}$)

(xii) $(5,3,1)$ ($b^2> 2c^2$,  $a^2> 2b^2$, $\frac{1}{a^2}+\frac{1}{c^2}>\frac{3}{b^2}$);

\end{figure}

In the remaining part of the paper we will find a parametrization for the intersection curves of each of the caustics with the ellipsoid (3). In p. 324, \cite{cayley} (see also p. 218, \cite{salmon}) parametrization in curvilinear coordinates $\xi$, $\eta$ of the caustics is given as
$$ 
(x(\xi,\eta))^2=-\frac{(a^2+\xi)^3(a^2+\eta)}{a^2\beta\gamma},
(y(\xi,\eta))^2=-\frac{(b^2+\xi)^3(b^2+\eta)}{a^2\alpha\gamma},
(z(\xi,\eta))^2=-\frac{(c^2+\xi)^3(c^2+\eta)}{a^2\alpha\beta},
$$
where $-a^2\le\xi,\eta\le-c^2$. By substituting these equalities in (3), we obtain
$$
\eta(\xi)=-\frac{1+\frac{(a^2+\xi)^3}{\beta\gamma a^2}+\frac{(b^2+\xi)^3}{\alpha\gamma b^2}+\frac{(c^2+\xi)^3}{\alpha\beta c^2}}{\frac{(a^2+\xi)^3}{\beta\gamma a^4}+\frac{(b^2+\xi)^3}{\alpha\gamma b^4}+\frac{(c^2+\xi)^3}{\alpha\beta c^4}}.
$$
Substituting these equalities in the above parametrization of the caustics and taking $\xi =t$ gives
$$ 
(x(t))^2=\frac{a^2(a^2 + t)^3((b^2 + c^2)t+3b^2c^2)}{(a^2-c^2)(a^2-b^2)((a^2b^2+ a^2c^2 + b^2c^2)t+3a^2b^2c^2)},
$$
$$
(y(t))^2=\frac{b^2(b^2 + t)^3((a^2 + c^2)t+3a^2c^2)}{(b^2-a^2)(b^2-c^2)((a^2b^2+ a^2c^2 + b^2c^2)t+3a^2b^2c^2)},
$$
$$
(z(t))^2=\frac{c^2(c^2 + t)^3( (a^2 + b^2)t+3a^2b^2)}{(c^2-a^2)(c^2-b^2)((a^2b^2+ a^2c^2 + b^2c^2)t+3a^2b^2c^2)},
$$
where $-a^2\le t\le-b^2$ and $-b^2\le t\le-c^2$ correspond to the two intersections of the caustics with the ellipsoid (3) (See Figure 10, cyan blue curves). This parametrization determines the curves on the ellipsoid (3) which separate the regions of the ellipsoid with different values for $n(A)$, thus completing the solution of Apollonius problem for triaxial ellipsoid.

\textbf{7. Conclusion} In the paper Apollonius problems for 2 dimensions (ellipse) and 3 dimensions (ellipsoid) were discussed. The number of concurrent normals of an ellipse (an ellipsoid) is dependent on the position of the point of concurrency with respect to caustics of the ellipse (the ellipsoid). The cases when the point of concurrency is on the ellipse (the ellipsoid), required the study of several different cases of intersections of the caustics with the given ellipse (ellipsoid). It would be interesting to generalize the results to 4 (see \cite{joet}) and higher dimensions.

\textbf{8. Appendix} The tangent lines of an ellipse, and the tangent lines and planes of an ellipsoid are much easier to study than the normals. For completeness, the problem on the number of concurrent tangent lines (planes) of an ellipse (ellipsoid) will be discussed here. Let us take point $A(X,Y)$ outside of the ellipse (1) and find point $B(x_1,y_1)$ and $C(x_2,y_2)$ on the ellipse (1) such that $AB$ and $AC$ are tangent lines of the ellipse (1) at $B$ and $C$, respectively. Since $AB$ and $AC$ have the same slopes as the ellipse (1) at the point $B$ and $C$, respectively, we have $\frac{y-Y}{x-X}=-\frac{b^2x}{a^2y}$, which can be written as $\frac{x(x-X)}{a^2}+\frac{y(y-Y)}{b^2}=0$. This is equation of ellipse through the points $O$ and $A$, with center at $\left(\frac{X}{2},\frac{Y}{2}\right)$ and semiaxes parallel to the semiaxes of the original ellipse. Its intersections with the original ellipse (1) are on the line $\frac{xX}{a^2}+\frac{yY}{b^2}=1$, which is obtained from subtraction of the equations of the ellipses. The coordinates of the points $B$ and $C$ are then determined by
$$
x_1=a\cdot \frac{\frac{X}{a}-\frac{Y}{b}\sqrt{\left(\frac{X}{a}\right)^2+\left(\frac{Y}{b}\right)^2-1}}{\left(\frac{X}{a}\right)^2+\left(\frac{Y}{b}\right)^2},\ y_1=b\cdot \frac{\frac{Y}{b}+\frac{X}{a}\sqrt{\left(\frac{X}{a}\right)^2+\left(\frac{Y}{b}\right)^2-1}}{\left(\frac{X}{a}\right)^2+\left(\frac{Y}{b}\right)^2},$$
$$x_2=a\cdot \frac{\frac{X}{a}+\frac{Y}{b}\sqrt{\left(\frac{X}{a}\right)^2+\left(\frac{Y}{b}\right)^2-1}}{\left(\frac{X}{a}\right)^2+\left(\frac{Y}{b}\right)^2},\ y_2=b\cdot \frac{\frac{Y}{b}-\frac{X}{a}\sqrt{\left(\frac{X}{a}\right)^2+\left(\frac{Y}{b}\right)^2-1}}{\left(\frac{X}{a}\right)^2+\left(\frac{Y}{b}\right)^2}.$$

\begin{figure}[htbp]
\centerline{\includegraphics[scale=.2]{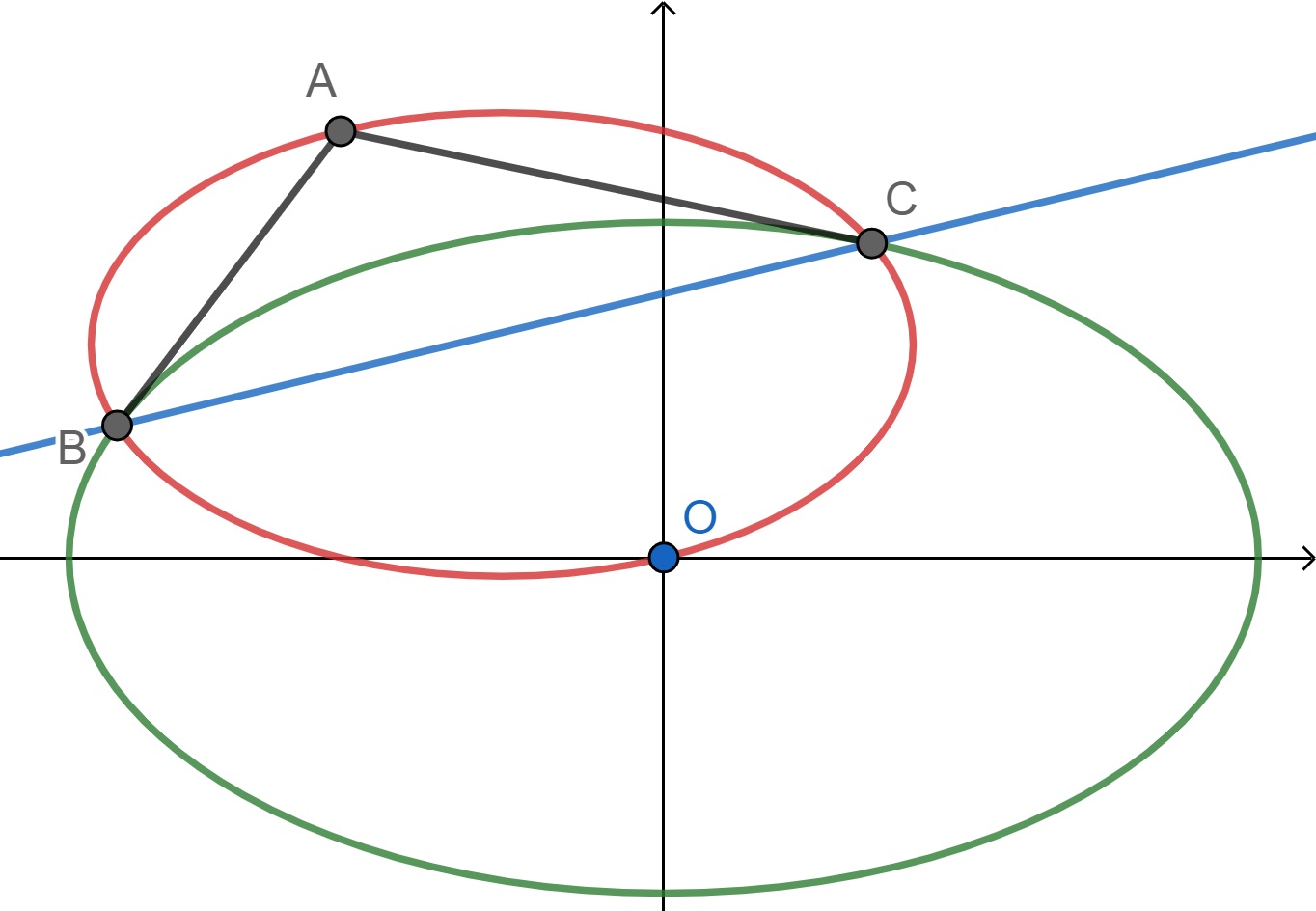}}
\label{fig12}
Figure 12: Ellipse $\frac{x^2}{a^2}+\frac{y^2}{b^2}=1$ (green), its 2 tangents (black), line $\frac{xX}{a^2}+\frac{yY}{b^2}=1$ (blue), and ellipse $\frac{x(x-X)}{a^2}+\frac{y(y-Y)}{b^2}=0$ (red).
\end{figure}
If we denote by $t(A)$ the total number of tangent lines of the ellipse passing through $A$, then $t(A)=2$ outside of the ellipse, $t(A)=0$ inside of the ellipse, and $t(A)=1$ on the ellipse.

Similarly, for the ellipsoid (3), if $A$ is outside of the ellipsoid (3), then the points $B$, such that $AB$ is tangent to the ellipsoid (3) can be determined by intersecting (3) with another ellipsoid $\frac{x(x-X)}{a^2}+\frac{y(y-Y)}{b^2}+\frac{z(z-Z)}{c^2}=0$. All these intersection points are on the plane $\frac{xX}{a^2}+\frac{yY}{b^2}+\frac{zZ}{c^2}=1$. The number of tangent lines $tl(A)$ and tangent planes $tp(A)$ of the ellipsoid (3), which pass through $A$ is infinity ($tl(A)=tp(A)=\infty$) for exterior points of the ellipsoid (3), $tl(A)=tp(A)=0$ for interior points, and $tl(A)=\infty$ and $tp(A)=1$ for the points of the ellipsoid (3) itself.

\begin{figure}[htbp]
\centerline{\includegraphics[scale=.4]{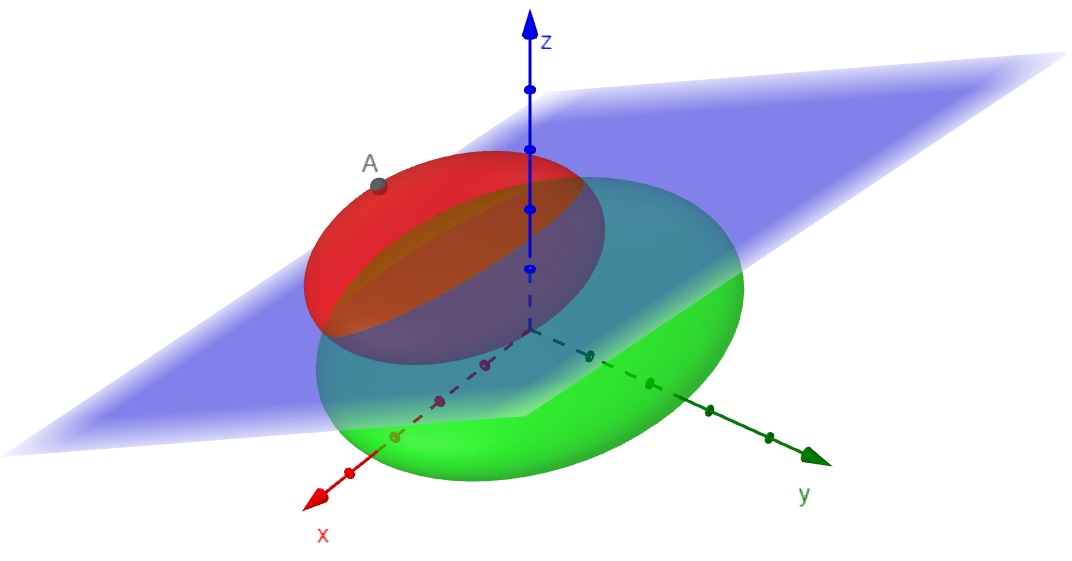}}
\label{fig13}
Figure 13: Ellipsoid $\frac{x^2}{a^2}+\frac{y^2}{b^2}+\frac{z^2}{c^2}=1$ (green), plane $\frac{xX}{a^2}+\frac{yY}{b^2}+\frac{zZ}{c^2}=1$ (blue), and ellipsoid $\frac{x(x-X)}{a^2}+\frac{y(y-Y)}{b^2}+\frac{z(z-Z)}{c^2}=0$ (red).
\end{figure}

\begin{figure}[htbp]
\centerline{\includegraphics[scale=.4]{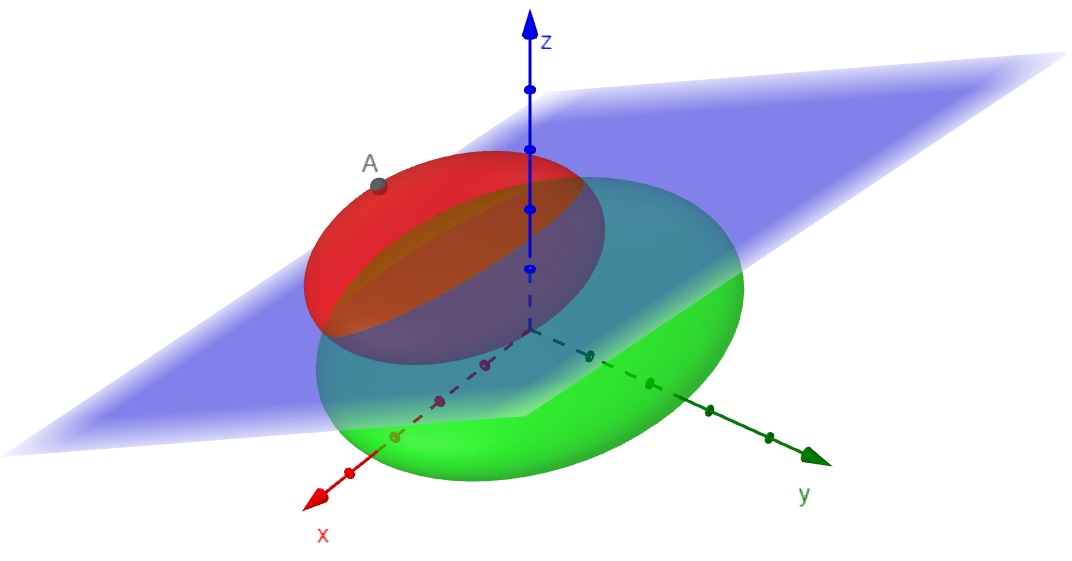}}
\label{fig13}
Figure 13: Ellipsoid $\frac{x^2}{a^2}+\frac{y^2}{b^2}+\frac{z^2}{c^2}=1$ (green), plane $\frac{xX}{a^2}+\frac{yY}{b^2}+\frac{zZ}{c^2}=1$ (blue), and ellipsoid $\frac{x(x-X)}{a^2}+\frac{y(y-Y)}{b^2}+\frac{z(z-Z)}{c^2}=0$ (red).
\end{figure}


\begin{thebibliography}{5}
\bibitem{abbena} E. Abbena, S. Salamon, A. Gray, \textit{Modern Differential Geometry of Curves and Surfaces with Mathematica}, 3rd Edition, Chapman and Hall/CRC, New York (2017). \url{https://doi.org/10.1201/9781315276038}

\bibitem{aliyev} Y. Aliyev, Cayley’s Centro-Surface: Old and New Attempts to Draw this Elusive Surface and Some New Ideas Around It, Contributed Talks: Maple in Education, Maple Conference, Nov 02-04 (2022). \url{https://www.maplesoft.com/mapleconference/2022/full-program.aspx} and \url{https://youtu.be/x-lhM7oja_Q}

\bibitem{apol} Apollonius of Perga, \textit{Conics Books V to VII, The Arabic Translation of the Lost Greek Original in the Version of the Banū Mūsā}, Gerald J. Toomer (ed.), Springer New York, NY (1990). \url{https://doi.org/10.1007/978-1-4613-8985-9}

\bibitem{arnold} V.I. Arnol'd, Astroidal geometry of hypocycloids and the Hessian topology of hyperbolic polynomials, \textit{Russian Math. Surveys}, \textbf{56}(6), 1019–1083 (2001). \url{https://doi.org/10.1070/RM2001v056n06ABEH000452}

\bibitem{arnold2} V.I. Arnold,  \textit{Astroidal geometry and topology} (in Russian), Summer school "Contemporary Mathematics", July 15-27, (2001). Lecture notes: \url{https://www.mccme.ru/dubna/2001/material/arnbook.pdf} Video recordings: 6 videos \url{https://youtube.com/playlist?list=PLdjLKCP6MkB0HAsjuDQcJvUbz9IhsXQuG}

\bibitem{arnold3} V.I. Arnold,  \textit{Catastrophe Theory}, Springer (1992). \url{https://doi.org/10.1007/978-3-642-58124-3}

\bibitem{afraj} V.S. Afrajmovich, Yu.S. Il'yashenko, L.P. Shil'nikov, V.I. Arnold, \textit{Dynamical Systems V, Bifurcation Theory and Catastrophe Theory}, Encyclopaedia of Mathematical Sciences, Springer (1994). \url{https://doi.org/10.1007/978-3-642-57884-7}

\bibitem{banc} T. Banchoff, \textit{Differential geometry and computer graphics}, Perspectives in Mathematics, Anniversary of Oberwolfach, Birkhauser Verlag, Basel, 43-60 (1984).

\bibitem{banc2} T.F. Banchoff, \textit{Beyond the Third Dimension, Geometry, Computer Graphics, and Higher Dimensions}, Scientific American Library, W H Freeman \& Co (1990). Online version: \url{https://www.math.brown.edu/tbanchof/Beyond3D.new/chapter7/s7_8.html}

\bibitem{bektas} S. Bektas, \textit{Geodesy II Applications on the surface of an ellipsoid} (Turkish), Atlas Akademi, Konya (2021).

\url{https://www.atlasakademiyayin.com/Stok/StokDetay/26071}

\bibitem{berger} M. Berger, \textit{Geometrie}, Cedic/ Fernand Nathan, Paris (1978).

\bibitem{berger3} M. Berger, \textit{Geometry Revealed
A Jacob's Ladder to Modern Higher Geometry}, Springer New York, NY (2010). \url{https://doi.org/10.1007/978-3-540-70997-8}

\bibitem{berger2} M. Berger, P. Pansu, J. Berry, X. Saint-Raymond, \textit{Problems in Geometry}, Springer New York, NY (1984). \url{https://doi.org/10.1007/978-1-4757-1836-2}

\bibitem{brill} L. Brill, Central Surface of a Paraboloid, Geometric Model,  No. 149. Ser. 1, No. 2a, The National Museum of American History (1892). \url{https://americanhistory.si.edu/collections/search/object/nmah_693993}

\bibitem{caspari} F. Caspari, Die Krümmungsmittelpunktsfläche des elliptischen Paraboloids, Dissert., Reimer, Berlin
(1875). \url{http://resolver.sub.uni-goettingen.de/purl?PPN310966825}

\bibitem{cayley2} A. Cayley, Note sur les normales d’une conique, \textit{Journal für die reine und angewandte Mathematik}, \textbf{56}, 182-185 (1859).  \url{http://resolver.sub.uni-goettingen.de/purl?PPN243919689_0056}

\bibitem{cayley} A. Cayley, On the centro-surface of an ellipsoid, \textit{Transactions of the Cambridge Philosophical Society}, \textbf{12}(1), 319-365 (1873). Also included in \textit{The collected mathematical papers of Arthur Cayley}, Vol. VIII, Cambridge University Press, Cambridge, 316-365 (1895). \url{http://name.umdl.umich.edu/ABS3153.0008.001}

\bibitem{cherrie} M. A. Cherrie, \textit{The conjugate locus in convex manifolds}, Doctoral Thesis, School of Mathematics \& Physics, University of Portsmouth, 2022. \url{https://researchportal.port.ac.uk/en/studentTheses/the-conjugate-locus-in-convex-manifolds}

\bibitem{getty} Centro-Surface. Ellipsoid., Getty Museum Collection, about 1860. \url{https://www.getty.edu/art/collection/object/107CZM}

\bibitem{clebsch} A. Clebsch, Ueber das Problem der Normalen bei Curven und Oberflächen der zweiten Ordnung, \textit{Journal für die reine und angewandte Mathematik}, \textbf{62}, 64-109 (1863). \url{http://eudml.org/doc/147884}

\bibitem{goet} Curvature centre point surface of the hyperboloid of one sheet, the triaxial ellipsoid and the elliptic paraboloid. Krümmungsmittelpunktsfläche Modellen: 238, 239, 242 und 345, Gypsum; Göttinger Sammlung mathematischer Modelle und Instrumente,  Georg-August-Universität Göttingen.

\url{https://sammlungen.uni-goettingen.de/objekt/record_DE-MUS-069123_238/1/-/}

\url{https://sammlungen.uni-goettingen.de/objekt/record_DE-MUS-069123_242/1/-/}

\url{https://sammlungen.uni-goettingen.de/objekt/record_DE-MUS-069123_239/1/-/}

\url{https://sammlungen.uni-goettingen.de/objekt/record_DE-MUS-069123_345/1/-/}

\bibitem{desbov} A. Desboves, \textit{Théorèmes et problèmes sur les normal}, Exercices pour les classes de mathématiques spéciales, Mallet-Bachelier, Paris (1861). \url{https://books.google.az/books?id=BVwGZKiMM7wC&pg=PP1#v=onepage&q&f=false}

\bibitem{dom} G. Domokos, Z. Lángi, T. Szabó, A topological classification of convex bodies, \textit{Geom. Dedicata} \textbf{182}, 95–116 (2016). \url{https://doi.org/10.1007/s10711-015-0130-4}

\bibitem{space} K.D. Drach, V. Komlev, Space evolute of an elliptic paraboloid and a one-sheeted hyperboloid of M. Schilling catalogue, Category: Space caustics of quadrics, Geometric Models Collection
of V.N. Karazin Kharkiv National University. \url{http://touch-geometry.karazin.ua/list/category-space-caustics-of-quadrics}

\bibitem{dyck} W. Dyck (ed.), \textit{Katalog mathematischer und mathematisch-physikalischer modelle, apparate und instrumente. Unter mitwirkung zahlreicher fachgenossen}, Deutsche Mathematiker-Vereinigung, C. Wolf \& Sohn, München (1892). \url{https://archive.org/details/katalogmathemat00goog/page/281/mode/2up}

\bibitem{dyck2} W. Dyck, Die Centralfläche des einschaligen Hyperboloids, \textit{Abhandlungen und Erläuterungen zu den mathematischen Modellen der Serien I-XII des Modellverlags}, unter Leitung von L. Brill, 13-18, Darmstadt (1877-1885). \url{https://opendigi.ub.uni-tuebingen.de/opendigi/BRILL#tab=struct&p=21}

\bibitem{eisen} L.P. Eisenhart, \textit{An introduction to differential geometry: with use of the tensor calculus}, Princeton University Press, Princeton (1947).

\url{https://archive.org/details/introductiontodi0000eise/mode/2up}

\bibitem{hubble} European Space Agency, \textit{Seeing quintuple}, Press release in Science \& Exploration, Space Science, Hubble Space Telescope, Aug 13, (2021). \url{https://www.esa.int/ESA_Multimedia/Images/2021/08/Seeing_quintuple}
Also posted as \textit{Hubble Sees Cosmic Quintuple}, Solar System and Beyond, 
NASA Hubble Mission Team, Goddard Space Flight Center, Aug 27, (2021). \url{https://www.nasa.gov/image-feature/goddard/2021/hubble-sees-cosmic-quintuple}

\bibitem{fischer} G. Fischer (ed.), \textit{Mathematical Models}, From the Collections of Universities and Museums, Photograph Volume and Commentary, Springer Spektrum (2017). \url{https://doi.org/10.1007/978-3-658-18865-8}

\bibitem{forsyth} A.R. Forsyth,  \textit{Lectures on the differential geometry of curves and surfaces}, Cambridge University Press, (1912). \url{https://archive.org/details/cu31924060289141/page/112/mode/2up}

\bibitem{geiser} C.F. Geiser, La normali dell’essoide, \textit{Annali Di Matematica Pura Ed Applicata} \textbf{1}, 317-328 (1867). \url{https://ia800708.us.archive.org/view_archive.php?archive=/22/items/crossref-pre-1909-scholarly-works/10.1007%252Fbf02360180.zip&file=10.1007%252Fbf02419181.pdf} \url{https://doi.org/10.1007/BF02419181}

\bibitem{hamf} W.G. Hamflett, 2014. Joachimsthal’s Theorem, \textit{The Mathematical Gazette}, \textbf{329} (299), 86 - 87 (1948). \url{https://doi.org/10.2307/3610710}

\bibitem{hart} F. Hartmann, R. Jantzen, Apollonius's Ellipse and Evolute Revisited, \textit{Convergence} August (2010). \url{https://www.maa.org/book/export/html/116798}

\bibitem{itoh} Ji. Itoh, K. Kiyohara, The Structure of the Conjugate Locus of a General Point on Ellipsoids and Certain Liouville Manifolds. Arnold Math J., \textbf{7}, 31–90 (2021). \url{https://doi.org/10.1007/s40598-020-00153-9}

\bibitem{joach2} F.  Joachimsthal, Über die Normalen der Ellipse und des Ellipsoids, \textit{Journal für die reine und angewandte Mathematik}, \textbf{26}, 172–178 (1843). \url{http://resolver.sub.uni-goettingen.de/purl?GDZPPN002143445}

\bibitem{joach} F.  Joachimsthal, Ueber die Anzahl reeller Normalen, welche von einem Punkte an ein Ellipsoid gezogen werden können, \textit{Journal für die reine und angewandte Mathematik}, \textbf{59}, 111-124 (1861). \url{https://doi.org/10.1515/crll.1861.59.111} 

\url{https://gdz.sub.uni-goettingen.de/download/pdf/PPN243919689_0059/PPN243919689_0059.pdf}

\bibitem{joet} A. Joets, R. Ribotta, Caustique de la surface ellipsoïdale
à trois dimensions, \textit{Experimental Mathematics} 8(1), 49-55 (1999). \url{https://projecteuclid.org/journals/experimental-mathematics/volume-8/issue-1/Caustique-de-la-surface-ellipso%C3%AFdale-%C3%A0-trois-dimensions/em/1047477111.full}

\bibitem{junker} H. Junker, \textit{Anschauungsmodelle in der mathematischen Forschung deutscher Gelehrter 1860–1877}, Dissert., Martin-Luther-Universität
Halle-Wittenberg (2023). \url{https://opendata.uni-halle.de/bitstream/1981185920/110975/1/Dissertation_MLU_2023_JunkerHannes.pdf}

\bibitem{kagan} V.F. Kagan, \textit{Foundations of the Theory of Surfaces} (in Russian), Part 1, OGIZ, Moscow (1947).

\bibitem{kummer} E. Kummer, Über ein Modell der Krümmungsmittelpunktsfl{\"a}che des dreiaxigen Ellipsoids, Sitzung der physikalisch-mathematischen Klasse vom 30. Juni 1862, \textit{Monatsberichten der königlichen preussischen Akademie der Wissenschaften},
426-428, (1862). \url{https://www.biodiversitylibrary.org/page/36507229#page/464/mode/1up}

\bibitem{lando} S.K. Lando, Something about Caustics (in Russian), Summer School "Contemporary Mathematics", July 18–29, Dubna, (2015). \url{http://www.mathnet.ru/present12099} \url{https://www.mccme.ru/dubna/2015/courses/lando.html}

\bibitem{legendre} A.M. Legendre, \textit{Traité des fonctions elliptiques et des intégrales eulériennes: avec des tables pour en faciliter le calcul numérique}, Volume 1, Huzard-Courcier, Paris (1825). \url{https://books.google.az/books?id=IC_vAAAAMAAJ&pg=PA350&source=gbs_toc_r&cad=3#v=onepage&q&f=false}

\bibitem{milne} J.J. Milne, The Conics of Apollonius, \textit{The Mathematical Gazette}, \textbf{1} (6), 49-55 (1895). \url{https://doi.org/10.2307/3604717}

\bibitem{monge} G. Monge, \textit{Application de l’Analyse à la Géométrie} (5e édition), Paris (1850). \url{https://gallica.bnf.fr/ark:/12148/bpt6k96431405/f344.item}

\bibitem{muller} R. Müller, \textit{Ueber eine gewisse Gleichung $2n$-ten Grades deren Specialfälle $n=2$ und $n=3$ beim Normalenproblem der Ellipse und des Ellipsoides auftreten}, Dissert., Berlin, G. Schade (1884). \url{https://edoc.hu-berlin.de/bitstream/handle/18452/786/27001.pdf?sequence=1}

\bibitem{nad} Z. Nádeník,  J. Bečvář, \textit{Moji učitelé geometrie}, Matfyzpress, Praha, (2011). (In Czech) \url{http://dml.cz/dmlcz/402172}

\bibitem{niem} R. Niemtschik, Einfaches Verfahren,
Normalen zu Fl{\"a}chen zweiter Ordnung durch ausserhalb liegende Punkte
zu ziehen, \textit{Sitzungsberichte der math.-natuwiss. Classe der kaiserlichen Akademie der Wissenschaften
Wien}, \textbf{58}, II. Abtheilung, 831–836 (1868). \url{https://books.google.az/books?id=s3M5AAAAcAAJ&pg=PP1#v=onepage&q&f=false}

\bibitem{nikolsky} S.M. Nikolsky, \textit{A Course of Mathematical Analysis}, Vol. 1, Mir Publishers, Moscow (1977). \url{https://archive.org/details/nikolsky-a-course-of-mathematical-analysis-vol-1-mir/page/203/mode/2up}

\bibitem{salmon} G. Salmon, On the equation of the surface of centres of an ellipsoid, \textit{The Quarterly Journal of Pure And Applied Mathematics}, \textbf{2}, 217–222 (1858). \url{https://babel.hathitrust.org/cgi/pt?id=uc1.$b417524&view=1up&seq=241}

\bibitem{schmidt} T. Schmidt, T. Treu et al., STRIDES: automated uniform models for 30 quadruply imaged quasars, \textit{Monthly Notices of the Royal Astronomical Society}, \textbf{518}(1), 1260–1300, (2023). \url{https://doi.org/10.1093/mnras/stac2235}

\bibitem{schroder} H. Schröder, \textit{Die Zentraflächen der Paraboloide und Mittelpunktsflächen zweiten Grades}, Dissert., Halle a. S., H. John (1913).
\url{http://resolver.sub.uni-goettingen.de/purl?PPN316295612}

\bibitem{tikh} V.M. Tikhomirov, \textit{Stories about maxima and minima}, Mathematical world series, \textbf{1}, AMS, (1990). \url{https://bookstore.ams.org/view?ProductCode=MAWRLD/1}


\bibitem{tsar} I.G. Tsar'kov, \textit{Mach Disks and Caustic Reflections, Caustics, Application to Astrophysics}, International Online Conference "Mathematical Physics, Dynamical Systems and Infinite-Dimensional Analysis" -MPDSIDA, Book of Abstracts, 201-202 (2021). \url{https://www.mathnet.ru/supplement/conf/1918/Abstracts-MPDSIDA2021.pdf}

\bibitem{tsar2} I.G. Tsarkov, Mach discs and caustic reflections, caustics, application in astrophysics, Proceedings of international scientific conference dedicated to 70th birthday of academician NAS of Tajikistan Shabozov M.Sh.: Modern problems of mathematical analysis and theory of functions, OOO ER-GRAF, Dushanbe, 138-144 (2022).

\bibitem{van} B.L. Van Der Waerden, \textit{Science Awakening I}, Springer Netherlands, (1975). \url{https://link.springer.com/book/10.1007/978-94-009-1379-0}

\bibitem{weath} C.E. Weatherburn, \textit{Differential Geometry Of Three Dimensions}, Vol.1,  Cambridge University Press, (1955). \url{https://archive.org/details/differentialgeom003681mbp/page/n175/mode/2up}

\bibitem{yang} W.-Ch. Yang, M.-L Lo, \textit{General Inverses in 2-D, 3-D, applications inspired by Technology}, The Electronic Journal of Mathematics and Technology, \textbf{2}(3), 243-260 (2008). \url{https://php.radford.edu/~ejmt/}

\bibitem{yoshi} S. Yoshizawa, A. Belyaev, H. Yokota, H. -P. Seidel, Fast and Faithful Geometric Algorithm for Detecting Crest Lines on Meshes, \textit{15th Pacific Conference on Computer Graphics and Applications (PG'07)}, Maui, HI, USA, 231-237 (2007). \url{https://doi.org/10.1109/PG.2007.24}

\bibitem{zeld} Ya.B. Zel'dovich, A.V. Mamaev, S.F. Shandarin, \textit{Laboratory observation of caustics, optical simulation of the motion of particles, and cosmology}, Phys. Usp., \textbf{26}(1), 77-83 (1983). \url{https://doi.org/10.1070/PU1983v026n01ABEH004307}

\bibitem{zeld2} Ya.B. Zel'dovich, A.D. Myskis, \textit{Elements of Mathematical Physics, A medium of non-interacting particles} (in Russian), Fizmatlit, Moscow (2008).
\end{thebibliography}
\end{document}